\documentclass[a4paper,11pt]{amsart}
\usepackage{amssymb}
\usepackage{amscd}
\usepackage{amsmath,amsthm}
\usepackage[colorlinks=true]{hyperref}
\usepackage{enumerate}
\usepackage{booktabs,multirow}
\usepackage{tikz}
\usepackage{rotating}
\usetikzlibrary{patterns}
\usetikzlibrary{decorations.pathreplacing}
\usetikzlibrary{calc,through}

\allowdisplaybreaks[1]
\numberwithin{figure}{section}
\numberwithin{equation}{section}

\title{Enumeration of labeled trees and Dyck tilings}
\author[K.~Shigechi]{Keiichi~Shigechi}
\email{k1.shigechi AT gmail.com}
\date{\today}

\newcommand\tikzpic[2]{
\raisebox{#1\totalheight}{
\begin{tikzpicture}
#2
\end{tikzpicture}
}}
\newtheorem{theorem}[figure]{Theorem}
\newtheorem{example}[figure]{Example}
\newtheorem{lemma}[figure]{Lemma}
\newtheorem{defn}[figure]{Definition}
\newtheorem{prop}[figure]{Proposition}

\newtheorem{remark}[figure]{Remark}
\begin{document}

\begin{abstract}
We study a partially ordered set of planar labeled rooted trees by use of 
combinatorial objects called Dyck tilings.
A generating function of the poset is factorized when the minimum element of the poset 
is $312$-avoiding and satisfies some extra condition. 
We define a cover relation on rational Dyck tilings by that of labeled 
trees, and show that increasing and decreasing labelings are dual to each other.
We consider two decompositions of a rational $(a,b)$-Dyck tiling: one is into $ab$ Dyck tilings 
and the other is into $a$ $(1,b)$-Dyck tilings.
In the first case, we show that the weight of the $(a,b)$-Dyck tiling is the sum of 
the weights of $ab$ Dyck tilings.
In the second case, we introduce a cover relation on $(1,b)$-Dyck tilings and obtain 
a poset of $(a,b)$-Dyck tilings by this decomposition.
\end{abstract}

\maketitle

\section{Introduction}
A planar rooted tree with $n$ edges is a planar tree with a distinguished node called the root.
By assigning an integer in $\{1,2,\ldots,n\}$ to an edge, we have a labeled tree.
We consider two special classes of labeled trees: one is increasing and the other is decreasing.
Since a tree is rooted, there is a unique sequence of edges from a leaf to the root. 
An increasing (resp. decreasing) tree is a tree such that the labels of the edges in a 
sequence from a leaf to the root are increasing (resp. decreasing) from the root to a leaf.
The enumeration of trees is a main research theme in combinatorial analysis.
For example, the number of increasing planar rooted tree with $n$ edges is 
given by $(2n-1)!!$ \cite{BerFlaSal92}. 
We refer the reader to \cite{BerLabLer91,Com74,GouJac83} for other results in combinatorial analysis.

We introduce a cover relation on labeled trees by a transposition of the labels.
The cover relation defines a partially ordered set (poset) of labeled trees, and we are 
interested in the generating function of labeled trees in this poset.
When the minimal element of the poset satisfies the conditions that it avoids the 
$312$-pattern, and satisfies another extra condition, the generating function has the factorized 
expression. Each factor in the expression corresponds to a simple generating function
characterized by a Young diagram.
In this paper, we consider an enumeration of such labeled trees, and the generating 
function of labeled trees through another combinatorial objects called Dyck tilings.
Basic definitions and properties of labeled trees are summarized in Section \ref{sec:LT}.

In Section \ref{sec:DT}, we summarize some basic facts about the cover-inclusive 
Dyck tilings following \cite{KW11,KMPW12,SZJ12}.
A Dyck tiling is a tiling of a region surrounded by the top and bottom Dyck paths by 
tiles called Dyck tiles.
A Dyck tile is characterized by a Dyck path, and we call a Dyck tile which 
is characterized by a Dyck path of size zero a trivial Dyck tile.
A Dyck tiling first appeared in the study of the Kazhdan--Lusztig 
polynomials for Grassmannian permutations in \cite{SZJ12}.
Dyck tilings also appeared in the context of the double-dimer model \cite{KW11} 
and spanning trees \cite{KW15}. 
In \cite{KMPW12}, a Dyck tiling above a Dyck path is identified with an 
increasing label of a planar rooted tree. 
A planar rooted tree $T$ corresponds to the bottom Dyck path, and 
its labeling on $T$ defines the top Dyck path and a Dyck tiling in the region 
between the top and bottom Dyck paths.
It is well-known that a planar rooted tree with $n$ edges is bijective to 
a Dyck path of size $n$.
Basic combinatorial tools for Dyck paths are thus labeled trees, 
Dyck tiling strips (DTS), and Hermite histories as used in \cite{KMPW12}.
The DTS operation and Hermite histories give bijections from 
a labeled trees to a Dyck tiling. 
In terms of labeled trees, they correspond to increasing 
and decreasing labeled trees respectively.
We are interested in the generating function of Dyck tilings in the poset 
of labeled trees introduced in Section \ref{sec:LT}.
We recall the $q$-analogue of the hook-length formula for the generating function 
for cover-inclusive Dyck tilings above a Dyck path, which appeared in \cite{KW11,KMPW12}.
The hook-length formula is a fraction of $q$-integers, and especially the expression 
is factorized.

In Section \ref{sec:LGV}, we study the generating function of trivial Dyck tiling
by use of the Lindstr\"om--Gessel--Viennot lemma \cite{GesVie85,Lin73}.
Since a tree structure gives a bottom Dyck path $\mu$, we consider a Dyck tiling 
consisting of only trivial Dyck tiles above $\mu$. 
Such a Dyck tiling is bijective to a Dyck path above $\mu$, and the generating 
function of such Dyck tilings can be expressed in terms of a determinant 
by the Lindstr\"om--Gessel--Viennot lemma \cite{GesVie85,Lin73}.
This is achieved by regarding a trivial Dyck tiling as a set of non-intersecting 
lattice paths. 
The generating function of trivial Dyck tilings is an important building 
block for the generating function of labeled trees studied in Section \ref{sec:gfLT}.

In Section \ref{sec:gfLT}, we will see that the generating function of labeled 
trees satisfies a recursive formula in general. 
The main theorem, Theorem \ref{thrm:GFLT312},  gives a factorized expression of 
the generating function of labeled trees such that the minimal element in the poset 
is $312$-avoiding and satisfies an extra condition.
A factor in the formula is given by the generating function of trivial Dyck tilings.
The $q$-analogue of the hook-length formula for the generating function mentioned 
in Section \ref{sec:DT} is a special case of a factorization in the main theorem.
The $312$-pattern avoidance of a labeled tree implies that the corresponding Dyck tiling 
does not have a non-trivial Dyck tiles. However, there are Dyck tilings which are $312$-containing 
but consist of only trivial Dyck tiles. In this sense, the pattern avoidance is an unavoidable condition
in Theorem \ref{thrm:GFLT312}.
We remark that the minimal element in the poset of labeled trees is $312$-avoiding does not mean that
other elements in the poset are $312$-avoiding. The poset may contain several $312$-containing labeled 
trees. A non-trivial Dyck tiling comes from a $312$-containing labeled tree.

In Section \ref{sec:crrev}, we also rephrase the cover relation of labeled trees in terms of a sequence of 
non-negative integers which encodes the same information as a labeled tree.
The poset of labeled trees introduced in Section \ref{sec:LT} is not in general a lattice.
To overcome this point, we introduce another cover relation on the integer sequences, 
hence on labeled trees. 
This new cover relation is weaker than the cover relation defined in Section \ref{sec:LT}.
We show that the poset of labeled trees under this new cover relation is a graded lattice.

In Section \ref{sec:rDT}, we study the rational $(a,b)$-Dyck tiling by generalizing the results of previous 
sections, especially Section \ref{sec:crrev}.
An $(a,b)$-Dyck tiling first studied in \cite{Shi21c} is a generalization of generalized Dyck tilings studied in \cite{JVK16}.
In the case of $(a,b)=(1,k)$, a Dyck tiling is shown to be 
bijective to a collection of sets of $k$ positive integers.
By defining a cover relation on an integer sequence introduced in Section \ref{sec:crrev},
we obtain a poset of rational Dyck tilings.  
A collection of sets can be naturally identified with the integer sequence, and with 
a labeled tree.
Recall we have two types of labels of a tree: one is increasing and the other is decreasing.
We show that these two labels are dual to each other, that is, one is associated to 
a $(1,k)$-Dyck tiling and the other is to a $(k,1)$-Dyck tiling.
We also study a decomposition of a $(1,k)$-Dyck tiling into a set of $k$ Dyck tilings where 
we use the decomposition studied in \cite{Shi21c,Shi21b}.
These $k$ Dyck tilings are shown to satisfy simple conditions.  
Finally, we study the $(a,b)$-Dyck tilings and its decomposition.
An $(a,b)$-Dyck tiling is decomposed into $ab$ Dyck tilings \cite{Shi21c,Shi21b}, and the weight of the tiling
is equal to the sum of the weights of $ab$ tilings.
As another application of the method obtained in previous sections, we consider 
the decomposition of a trivial $(a,b)$-Dyck tiling into $a$ trivial $(1,b)$-Dyck 
tilings. We introduce a cover relation on the vertical Hermite histories of these 
$(1,b)$-Dyck tilings, and obtain a poset of decomposed $(a,b)$-Dyck tilings.

\subsection*{Notation}
The quantum integer $[n]$ is given by $[n]:=(1-q^{n})/(1-q)$.	

\section{Labeled trees}
\label{sec:LT}
Let $T$ be a planar rooted tree with $n$ edges, 
and $E_{n}(T)$ the set of edges in $T$.
A label $L^{\uparrow}$, (resp. $L^{\downarrow}$) of $T$ is called 
{\it increasing label} (resp. {\it decreasing label}) if the map
$L^{\uparrow}$ (resp. $L^{\downarrow}$) $:E_{n}(T)\rightarrow[n]:=\{1,2,\ldots,n\}$
satisfies the following two conditions:
\begin{enumerate}
\item Each integer in $[n]$ is assigned to exactly one edge in $E_{n}(T)$.
\item The labels in $L^{\uparrow}$ (resp. $L^{\downarrow}$) are increasing 
(resp. decreasing) from the root to a leaf in $T$.
\end{enumerate}

Given an edge $e\in E_{n}(T)$, we have a unique sequence of edges from $e$ 
to an edge $e_{\mathfrak{r}}$ which is adjacent to the root $\mathfrak{r}$.
We denote this sequence by 
\begin{align}
\label{eq:defp}
p(e):=(e=e_0\rightarrow e_1\rightarrow e_2\rightarrow\dots\rightarrow e_{\mathfrak{r}}),
\end{align}
where $e_{i}\rightarrow e_{i+1}$ means that the edge $e_{i+1}$ is the parent edge 
of $e_{i}$.

Let $p(e)$ and $p(e')$ be two sequences of edges from $e$ and $e'$.
We define the set $p_{\cap}(e,e'):=p(e)\cap p(e')$ as the 
intersection of the two sets $p(e)$ and $p(e')$.
Note that $p_{\cap}(e,e')=\emptyset$ if two edges $e$ and $e'$ in 
different subtrees of the root $\mathfrak{r}$.
The set $p_{\cap}(e,e')\neq\emptyset$ if $e$ and $e'$ belong to 
the same subtree of $\mathfrak{r}$ since $e_{\mathfrak{r}}\in p(e),p(e')$.

We say that the edge $e$ is strictly right (resp. left) to the edge $e'$ 
if $p(e)\setminus p_{\cap}(e,e')\neq\emptyset$, $p(e')\setminus p_{\cap}(e,e')\neq\emptyset$ 
and $e$ is right (resp. left) to $e'$ in the tree $T$.
If $p_{\cap}(e,e')=p(e')$, we say that the edge $e$ is weakly right or weakly left to the edge 
$e'$.

Let $e_{i}, 1\le i\le 3$, be three edges in $T$ such that $e_{i}$ is strictly right 
to $e_{i+1}$ for $1\le i\le 2$.
Denote by $L(e_{i})$ the label $L$ of the edge $e_{i}$.

\begin{defn}
Let $e_{i}, 1\le i\le 3$, be edges as above. 
A labeling $L$ is said to be $312$-avoiding if and only if 
there exist no three edges $e_{i}, 1\le i\le 3$, such that
\begin{align}
L(e_2)<L(e_3)<L(e_1).
\end{align}
\end{defn}

We introduce one more notion for a node.
\begin{defn}
Let $n$ be a node in a tree $T$. The node $n$ is called a branch point
if $n$ has at least two edges below it.
\end{defn}

We will construct a poset of labeled trees as follows.
Let $L$ and $L'$ be two decreasing labels of the tree $T$, and $e_{l}$ and $e_{r}$
be two edges in $E_{n}(T)$ such that $e_{l}$ is strictly left to $e_{r}$.
Note that the two edges $e_{l}$ and $e_{r}$ may not be leaves in the tree.
In this notation, we define a covering relation on the labels.

\begin{defn}
\label{def:coverL}
We say that $L'$ covers $L$ (denoted $L\lessdot$ L') if and only if $L$ and $L'$ 
satisfy the following conditions:
\begin{enumerate}
\item $L'$ is increasing (resp. decreasing) if $L$ is increasing (resp. decreasing).
\item An edge $e\in E_{n}(T)\setminus\{e_l,e_r\}$ satisfies $L(e)=L'(e)$.
\item The labels satisfy $L(e_l)<L(e_r)$, $L'(e_r)=L(e_l)$ and $L'(e_l)=L(e_r)$.
\item There exists no edge $e_{c}$ such that 
\begin{enumerate}
\item The edge $e_c$ is strictly right to $e_{l}$ and strictly left to $e_r$.
\item The label $L(e_c)\in[L(e_l),L(e_r)]$.
\end{enumerate}
\end{enumerate}
\end{defn}
We write $L\le L'$ if there exists an unrefineable sequence 
\begin{align}
L=L_{0}\lessdot L_1 \lessdot\ldots\lessdot L_{p}=L',
\end{align}
where $p\ge0$.

\begin{defn}
\label{defn:posetlt}
We denote the poset of labeled trees whose minimal element 
is $L_{0}$ by $(L_{0},\lessdot)$.
\end{defn}

\begin{example}
We consider the following poset for labeled trees.
The left-most decreasing labeled tree in Figure \ref{fig:LTHasse} 
is the minimal element in the poset.	
\begin{figure}[ht]
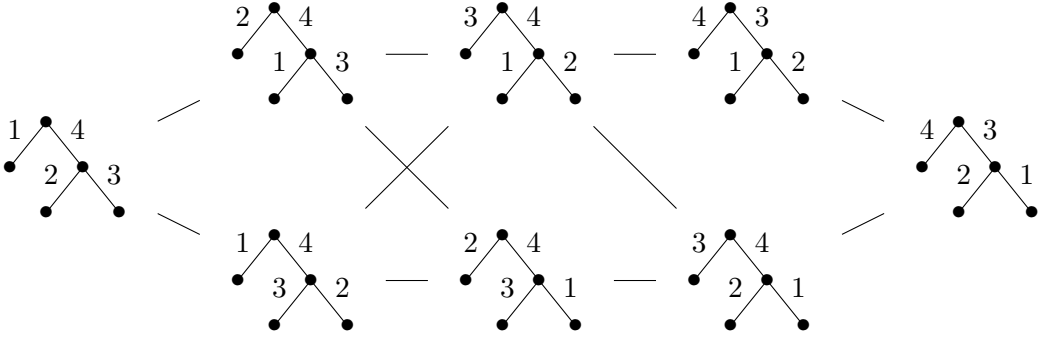

\tikzpic{-0.5}{	

\node(l1)at(0,0){
\tikzpic{-0.5}{[scale=0.6]
\draw(0,0)--(-0.8,-1)(0,0)--(0.8,-1);
\draw(0.8,-1)--(0,-2)(0.8,-1)--(1.6,-2);
\draw(-0.3,-0.6)node[anchor=south east]{$1$}(0.3,-0.6)node[anchor=south west]{$4$}
(0.5,-1.6)node[anchor=south east]{$2$}(1.1,-1.6)node[anchor=south west]{$3$};
\draw(0,0)node{$\bullet$}(-0.8,-1)node{$\bullet$}(0.8,-1)node{$\bullet$}
(0,-2)node{$\bullet$}(1.6,-2)node{$\bullet$};
}
};
\node(l2)at(3,1.5){
\tikzpic{-0.5}{[scale=0.6]
\draw(0,0)--(-0.8,-1)(0,0)--(0.8,-1);
\draw(0.8,-1)--(0,-2)(0.8,-1)--(1.6,-2);
\draw(-0.3,-0.6)node[anchor=south east]{$2$}(0.3,-0.6)node[anchor=south west]{$4$}
(0.5,-1.6)node[anchor=south east]{$1$}(1.1,-1.6)node[anchor=south west]{$3$};
\draw(0,0)node{$\bullet$}(-0.8,-1)node{$\bullet$}(0.8,-1)node{$\bullet$}
(0,-2)node{$\bullet$}(1.6,-2)node{$\bullet$};
}
};
\node(l3)at(3,-1.5){
\tikzpic{-0.5}{[scale=0.6]
\draw(0,0)--(-0.8,-1)(0,0)--(0.8,-1);
\draw(0.8,-1)--(0,-2)(0.8,-1)--(1.6,-2);
\draw(-0.3,-0.6)node[anchor=south east]{$1$}(0.3,-0.6)node[anchor=south west]{$4$}
(0.5,-1.6)node[anchor=south east]{$3$}(1.1,-1.6)node[anchor=south west]{$2$};
\draw(0,0)node{$\bullet$}(-0.8,-1)node{$\bullet$}(0.8,-1)node{$\bullet$}
(0,-2)node{$\bullet$}(1.6,-2)node{$\bullet$};
}
};
\node(l4)at(6,1.5){
\tikzpic{-0.5}{[scale=0.6]
\draw(0,0)--(-0.8,-1)(0,0)--(0.8,-1);
\draw(0.8,-1)--(0,-2)(0.8,-1)--(1.6,-2);
\draw(-0.3,-0.6)node[anchor=south east]{$3$}(0.3,-0.6)node[anchor=south west]{$4$}
(0.5,-1.6)node[anchor=south east]{$1$}(1.1,-1.6)node[anchor=south west]{$2$};
\draw(0,0)node{$\bullet$}(-0.8,-1)node{$\bullet$}(0.8,-1)node{$\bullet$}
(0,-2)node{$\bullet$}(1.6,-2)node{$\bullet$};
}
};
\node(l5)at(6,-1.5){
\tikzpic{-0.5}{[scale=0.6]
\draw(0,0)--(-0.8,-1)(0,0)--(0.8,-1);
\draw(0.8,-1)--(0,-2)(0.8,-1)--(1.6,-2);
\draw(-0.3,-0.6)node[anchor=south east]{$2$}(0.3,-0.6)node[anchor=south west]{$4$}
(0.5,-1.6)node[anchor=south east]{$3$}(1.1,-1.6)node[anchor=south west]{$1$};
\draw(0,0)node{$\bullet$}(-0.8,-1)node{$\bullet$}(0.8,-1)node{$\bullet$}
(0,-2)node{$\bullet$}(1.6,-2)node{$\bullet$};
}
};
\node(l6)at(9,1.5){
\tikzpic{-0.5}{[scale=0.6]
\draw(0,0)--(-0.8,-1)(0,0)--(0.8,-1);
\draw(0.8,-1)--(0,-2)(0.8,-1)--(1.6,-2);
\draw(-0.3,-0.6)node[anchor=south east]{$4$}(0.3,-0.6)node[anchor=south west]{$3$}
(0.5,-1.6)node[anchor=south east]{$1$}(1.1,-1.6)node[anchor=south west]{$2$};
\draw(0,0)node{$\bullet$}(-0.8,-1)node{$\bullet$}(0.8,-1)node{$\bullet$}
(0,-2)node{$\bullet$}(1.6,-2)node{$\bullet$};
}
};
\node(l7)at(9,-1.5){
\tikzpic{-0.5}{[scale=0.6]
\draw(0,0)--(-0.8,-1)(0,0)--(0.8,-1);
\draw(0.8,-1)--(0,-2)(0.8,-1)--(1.6,-2);
\draw(-0.3,-0.6)node[anchor=south east]{$3$}(0.3,-0.6)node[anchor=south west]{$4$}
(0.5,-1.6)node[anchor=south east]{$2$}(1.1,-1.6)node[anchor=south west]{$1$};
\draw(0,0)node{$\bullet$}(-0.8,-1)node{$\bullet$}(0.8,-1)node{$\bullet$}
(0,-2)node{$\bullet$}(1.6,-2)node{$\bullet$};
}
};
\node(l8)at(12,0){
\tikzpic{-0.5}{[scale=0.6]
\draw(0,0)--(-0.8,-1)(0,0)--(0.8,-1);
\draw(0.8,-1)--(0,-2)(0.8,-1)--(1.6,-2);
\draw(-0.3,-0.6)node[anchor=south east]{$4$}(0.3,-0.6)node[anchor=south west]{$3$}
(0.5,-1.6)node[anchor=south east]{$2$}(1.1,-1.6)node[anchor=south west]{$1$};
\draw(0,0)node{$\bullet$}(-0.8,-1)node{$\bullet$}(0.8,-1)node{$\bullet$}
(0,-2)node{$\bullet$}(1.6,-2)node{$\bullet$};
}
};
\draw(l1)to(l2)to(l4)to(l6)to(l8) (l1)to(l3)to(l5)to(l7)to(l8)(l3)to(l4);
\draw(l2)to(l5)(l4)to(l7);
}
\caption{Hasse diagram of the poset of labeled trees.}
\label{fig:LTHasse}
\end{figure}
\end{example}

\begin{remark}
Note that the poset in Figure \ref{fig:LTHasse} is not a lattice since there are two 
labeled trees which cover the two labeled tree in the second column.
In Section \ref{sec:crrev}, we introduce a cover relation which is weaker than 
the cover relation in Definition \ref{def:coverL}, and show that the new poset
is a graded lattice.
\end{remark}

Let $L$ be a label on $T$.
\begin{defn}
We define the word $\omega(L):=(\omega_{1},\ldots,\omega_{n})$ (resp. $\omega'(L)$) to be 
the post-order (resp. pre-order) word (from left to right) of the label $L$
with respect to $T$.
\end{defn}
In the above definition, the post-order means as follows. We first read the label on the left-most 
edge $e$ connected to a leaf in $T$. This label gives an element $\omega_{1}$.
We denote by $T\setminus\{e\}$ the tree which is obtained from $T$ by deleting the edge $e$.
Then, we read the label on the left-most edge $e'$ connected to a leaf in $T\setminus\{e\}$.
The label of $e'$ gives $\omega_{2}$.
We continue this process until we visit all the edges in $T$. 
Then, the sequence of the labels $\omega_1, \omega_2, \ldots, \omega_n$ gives $\omega(L)$.

Similarly, the pre-order means as follows.
We fist read the label on the left-most edge $e_1$ connected to the root in $T$. This label
gives an element $\omega'_1$.
Then, we read the label of the left-most edge $e_2$ connected to $e_1$, and this label 
gives $\omega'_2$.
In other words, we first visit the left-most subtree. 
Then, we visit the left-most subtree in the remaining tree.
The sequence of labels gives $\omega'(L)$.

By construction, the word $\omega(L)$ (or $\omega'(L)$) can be naturally identified 
to a permutation in $[n]$.
By regarding $\omega(L)$ (or $\omega'(L)$) as a permutation $w$, we define the 
inversion number $\mathrm{inv}(w)$ by 
\begin{align*}
\mathrm{inv}(w)
:=\sum_{i=1}^{n}\#\{1\le j\le i-1 : w_{j}<w_{i}\}.
\end{align*}
For example, we have $\mathrm{inv}(\omega)=2$ if $\omega=52431$.
Note that we have $\mathrm{inv}(w)=0$ for the longest permutation 
$w=(n,n-1,\cdots,1)$ and this definition is different from the standard 
notion of the inversion number.

We define a generating function of labeled trees as follows.
\begin{defn}
\label{defn:ZLT}
Let $L$ be a label of $T$. Then, we define 
the generating function $Z(L,T)$ by
\begin{align}
Z(L,T):=\sum_{L\le L'}q^{\mathrm{inv}(\omega(L'))}.
\end{align} 
\end{defn}
In Section \ref{sec:DT}, we consider the generating function $Z(L,T)$ where 
the word of the increasing labeled tree $L$ is $1,\ldots,n$ by pre-order 
from left to right. This case is studied in \cite{KW11,KMPW12}.

In Section \ref{sec:gfLT}, we relax the condition on $L$. However, we impose 
the condition that the labeled tree is $312$-avoiding. 
This pattern avoidance is necessary for the generating function to have 
a factorization property. 

\section{Dyck tilings}
\label{sec:DT}
\subsection{Dyck paths and planar rooted trees}
We begin with the correspondence between a tree $T$ and a Dyck path $\lambda(T)$.
A {\it Dyck path} $\lambda=\lambda_1\ldots \lambda_{2n}$ of size $n$ is 
a sequence of $n$ up ($U$) steps and $n$ down ($D$) steps such that 
\begin{align}
\label{eq:condDyck}
\#\{\lambda_j=U : 1\le j\le i\}\ge \#\{\lambda_j=D : 1\le j\le i\},
\end{align} 
for all $1\le i\le 2n$.

\begin{example}
We have five Dyck paths of size $3$:
\begin{align*}
UUUDDD, \quad UUDUDD, \quad UDUUDD, \quad UUDDUD, \quad UDUDUD.
\end{align*}
\end{example}

Let $\lambda$ be a Dyck path of size $n$.
Suppose that $\lambda_{j}=D$ with $2\le j\le n$.
From the condition (\ref{eq:condDyck}), there exists a unique maximal integer $1\le i\le j-1$
such that 
\begin{enumerate}
\item $\lambda_{i}=U$,
\item $\#\{\lambda_{k}=U : i\le k\le j\}=\#\{\lambda_{k}=D : i\le k\le j\}$.
\end{enumerate}
We say that the pair $(i,j)$ is a {\it chord pair} if it satisfies the condition above.

Consider $2n$ points which are enumerated by $1,2,\ldots, 2n$ from left to right.
We connect the two points labeled $i$ and $j$ by an arch if the pair $(i,j)$ is 
a chord pair.
Since $\lambda$ is of size $n$, the diagram consists of $n$ arches. 
By definition of chord pairs, these $n$ arches are non-crossing.
We call a diagram with $n$ non-crossing arches a chord diagram.
We denote by $\mathrm{Arc}(\lambda)$ the set of chord pairs in the Dyck path $\lambda$. 
We define the length of a chord pair $a:=(i,j)$ by $|a|=(j-i+1)/2$.

A planar rooted tree $T$ corresponding to a Dyck path $\lambda$ as follows.
Note that we have a bijection between Dyck paths and non-crossing chord diagrams.
Below, we often identify a Dyck path with a chord diagram.

We call a Dyck path $\lambda$ of size $n$ {\it prime} if $(1,2n)$ is a chord pair 
in $\lambda$.
Then, any Dyck path $\lambda$ can be expressed as a concatenation of prime 
Dyck paths, i.e., 
$\lambda=\lambda_1\circ\lambda_2\circ\ldots\circ\lambda_{p}$ where 
Dyck paths $\lambda_{i}$, $1\le i\le p$, are all prime. 

The planar rooted tree $T(\lambda)$ for a Dyck path $\lambda$ is recursively 
obtained by the following way.
\begin{enumerate}
\item If $\lambda=U\lambda'D$ is a prime Dyck path, the tree $T(\lambda)$ 
is obtained from $T(\lambda')$ by attaching an edge above the root 
of $T(\lambda')$.
\item Suppose $\lambda=\lambda_1\circ\lambda_2\circ\ldots\circ\lambda_{p}$. 
Then, the tree $T(\lambda)$ is obtained from $T(\lambda_{i}), 1\le i\le p$, 
by gluing their roots into a single root.
\item A empty Dyck path corresponds to a planar rooted tree which consists of 
only the root.
\end{enumerate}

\begin{remark}
In the above recursive construction, the number of edges $T(\lambda)$ is one plus 
the number of edges in $T(\lambda')$. 
On the other hand, the number of edges remains the same before and after the second 
operation (2).
These imply that an edge in $T(\lambda)$ is bijective to a chord pair in $\lambda$. 
\end{remark}

The above recursive construction of a tree $T$ from a Dyck path $\lambda$ can 
be easily reversed. Therefore, we have a bijection between trees $T$ with $n$ edges 
and Dyck paths of size $n$.
By this bijection, we identify a tree $T$ with a Dyck path, and vice versa.
 
We denote by $\mathrm{Arch}(T)=\mathrm{Arch}(\lambda)$ the set of 
chord pairs in $\lambda$, or equivalently $T$.

\subsection{Dyck tilings}
We introduce the notion of Dyck tilings following \cite{KW11,KMPW12,SZJ12}.
In this work, we focus on only cover-inclusive Dyck tilings.

Let $\lambda$ be a Dyck paths of size $n$.
We introduce a notion of Dyck tiling above a Dyck path $\lambda$.
In this paper, we consider only cover-inclusive Dyck tilings.
Then, we introduce two operations, called Dyck tiling strip and Hermite history, 
which yield a Dyck tiling from an increasing or decreasing label following \cite{KMPW12}.

A {\it Dyck tile} of size $n$ is a connected ribbon (a skew shape which does not 
contain a two-by-two square) such that the center of boxes (which consists of the ribbon) 
form a Dyck path of size $n$.
Thus, a Dyck tile of size $n$ consists of $2n+1$ unit boxes.
We say that  a Dyck tile of size zero is a trivial Dyck tile, and a Dyck tile of size $n\ge1$
is a non-trivial Dyck tile.

Let $\lambda$ and $\mu$ be two Dyck paths such that $\mu$ is above $\lambda$.
Then, a Dyck tiling is a tiling by Dyck tiles in the region between $\lambda$ and $\mu$.
A cover-inclusive Dyck tiling is a Dyck tiling such that the sizes of Dyck tiles are weakly 
decreasing from bottom to top. 
More precisely, let $d_1$ be a Dyck tile. If we move $d_1$ downward, then it is 
contained in another Dyck tile $d_2$ of below the bottom Dyck path.
In this paper, a Dyck tiling always means a cover-inclusive Dyck tiling.

An example of a Dyck tiling is given in Figure \ref{fig:exDT}.
This Dyck tiling has four trivial Dyck tiles, and two non-trivial Dyck tiles 
of size one and two.
\begin{figure}[ht]
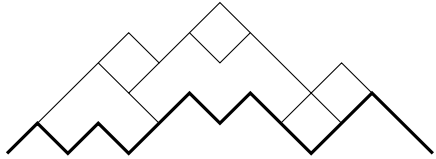

\tikzpic{-0.5}{[scale=0.4]
\draw[very thick](0,0)--(1,1)--(2,0)--(3,1)--(4,0)--(6,2)--(7,1)--(8,2)--(10,0)--(12,2)--(14,0);
\draw(1,1)--(4,4)--(5,3)--(7,5)--(10,2)--(11,3)--(12,2);
\draw(3,3)--(5,1)(4,2)--(5,3)(6,4)--(7,3)--(8,4)(11,1)--(10,2)--(9,1);
}
\caption{Example of a Dyck tiling}
\label{fig:exDT}
\end{figure}

\subsubsection{DTS operation}
We introduce a combinatorial operation called {\it Dyck tiling strips} (DTS) operation following \cite{KMPW12}.
Given an increasing label $L$ of size $n$, we will construct 
a Dyck tiling $\mathcal{DTS}(L;\lambda)$ above a Dyck path $\lambda$ or equivalently a tree $T$.
Let $L'$ be the increasing label obtained from $L$ by deleting 
the edge $e(n)$ with label $n$. This deletion is well-defined since 
$L$ is increasing implies that the edge $e(n)$ is connected to a leaf in 
$T$.
By the correspondence between a tree $T$ and a Dyck path $\lambda$, the edge $e(n)$ in $T$
corresponds to a pattern $UD$ in $\lambda$.
Suppose we have a Dyck tiling $\mathcal{DTS}(L';\lambda')$.
The Dyck path $\lambda$ can be obtained from $\lambda'$ by inserting a Dyck path $UD$ 
at a point $p'$ of $\lambda'$.
Since $UD$ form a peak in $\lambda$, the insertion of $UD$ gives a new peak $p$ in $\lambda$.

If $\mathcal{DTS}(L';\lambda')$ has a Dyck tile $d$ above the point $p'$ in $\lambda'$,
we increase the size of the tile $d$ by one such that its shape is compatible with 
the Dyck path $\lambda$.
Let $D'$ be the Dyck tiling obtained from $\mathcal{DTS}(L';\lambda')$.
In Figure \ref{fig:insUD}, we show an example of insertion of $UD$ to a Dyck tiling.
We insert $UD$ at the point $p'$ (it is a peak in $\lambda'$), and this gives 
a new peak $p$ in $\lambda$.
Two Dyck tilings of size zero and one in the left picture 
are enlarged to the Dyck tilings of size one and two in the right picture.
\begin{figure}[ht]
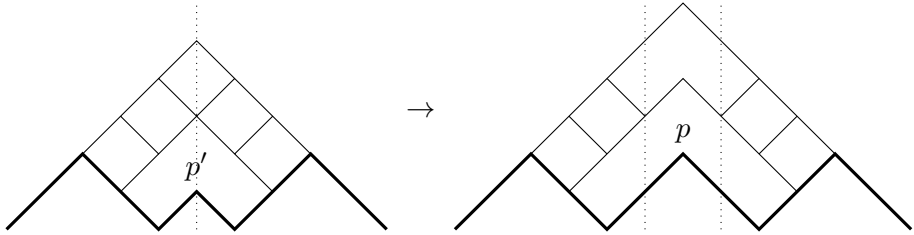

\begin{align*}
\tikzpic{-0.5}{[scale=0.5]
\draw[very thick](0,0)--(2,2)--(4,0)--(5,1)node[anchor=south]{$p'$}--(6,0)--(8,2)--(10,0);
\draw(2,2)--(5,5)--(8,2)(3,1)--(5,3)--(7,1)(3,3)--(4,2)(4,4)--(5,3)--(6,4)(6,2)--(7,3);
\draw[dotted](5,0)--(5,6);
}\rightarrow
\tikzpic{-0.5}{[scale=0.5]
\draw[very thick](0,0)--(2,2)--(4,0)--(6,2)node[anchor=south]{$p$}--(8,0)--(10,2)--(12,0);
\draw(2,2)--(6,6)--(10,2)(3,1)--(6,4)--(9,1);
\draw(3,3)--(4,2)(4,4)--(5,3)(9,3)--(8,2)(8,4)--(7,3);
\draw[dotted](5,0)--(5,6)(7,0)--(7,6);	
}
\end{align*}
\caption{Insertion of $UD$ to a Dyck tiling}
\label{fig:insUD}
\end{figure}

If we have up steps on the top Dyck path of $D'$ right to the peak $p$, 
we put a trivial Dyck tile $d$ to each up step such that the southeast edge of $d$ is attached to 
the up step.
In Figure \ref{fig:addtD}, we give an example of addition of trivial tiles.
There are two up steps right to the peak $p$. Therefore, we add two Dyck tiles of size zero 
to these two up steps. The added boxes are marked by $*$.
\begin{figure}[ht]
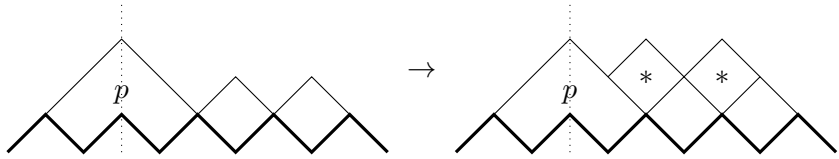

\begin{align*}
\tikzpic{-0.5}{[scale=0.5]
\draw[very thick](0,0)--(1,1)--(2,0)--(3,1)node[anchor=south]{$p$}--(4,0)--(5,1)--(6,0)--(7,1)--(8,0)--(9,1)--(10,0);
\draw(1,1)--(3,3)--(5,1)--(6,2)--(7,1)--(8,2)--(9,1);
\draw[dotted](3,0)--(3,4);
}
\rightarrow
\tikzpic{-0.5}{[scale=0.5]
\draw[very thick](0,0)--(1,1)--(2,0)--(3,1)node[anchor=south]{$p$}--(4,0)--(5,1)--(6,0)--(7,1)--(8,0)--(9,1)--(10,0);
\draw(1,1)--(3,3)--(5,1)--(6,2)--(7,1)--(8,2)--(9,1);
\draw[dotted](3,0)--(3,4);
\draw(4,2)--(5,3)--(6,2)--(7,3)--(8,2);
\draw(5,2)node{$*$}(7,2)node{$*$};
}
\end{align*}
\caption{Addition of trivial Dyck tiles}
\label{fig:addtD}
\end{figure}

For $n=1$, we have a unique Dyck path $UD$ and its trivial Dyck tiling above it, i.e., 
a Dyck tiling consisting of a Dyck path $UD$. 

We call the operation to construct the Dyck tiling $\mathcal{DTS}(L;\lambda)$ 
from $L$ and $\lambda$ the {\it Dyck tiling strip} (DTS) operation.

\subsubsection{Hermite history}
We introduce the notion of Hermite histories following \cite[Section 3]{KMPW12}.
Given a Dyck tiling $D$ above a Dyck paths $\lambda$ of size $n$, we will define a sequence of 
integers $\mathbf{h}:=(h_1,\ldots,h_{n})$, and construct a decreasing label on the tree $T(\lambda)$.
Recall that a Dyck tile of size $m$ has the weight $m+1$ for $m\ge1$.
We consider a line passing through Dyck tiles as in Figure \ref{fig:Hh}. More precisely, 
we draw a line from the center of an up step such that it passes through a Dyck tile from the south-east edge 
to the left-most north-west edge. 
Let $c(u)$ be the coordinate of an up step $u$.
If there exists $p$ Dyck tiles above an up step $u_1$ such that the lines in the Dyck tiles associate to other edges, 
we draw a line $l$ from the point $u_2$ satisfying $c(u_2)=c(u_1)+(0,p)$.
Note that the point $u_2$ is always on an edge of Dyck tile which is up-right. 
Then, we say that the line $l$ is associated to the up step $u_1$.
The collection of the lines associated to up steps is called an {\it Hermite history}.
Given an Hermite history, we assign an integer $h_{n+1-i}$ to the $i$-th up step from right as follows.
The value $h_{n+1-i}$ is the sum of the weight of Dyck tiles in the line $l$ associated to the $i$-th 
up step.
By regarding $\mathbf{h}$ as a Lehmer code of a permutation, the sequence $\mathbf{h}$ defines 
a post-order word $\omega_{h}$ reading from right to left.
Since the Dyck path $\lambda$ defines a tree $T$, and we have a word $\omega_h$, 
these determine a decreasing label of $T$.
In this way, we have a correspondence between a Dyck tiling and a decreasing label through 
the Hermite history.
It is easy to see that given a decreasing label $L^{\downarrow}$ and a tree $T$, 
we have a unique Dyck tiling whose Hermite history gives the Lehmer code $\mathbf{h}$ 
of the post-order word of $L^{\downarrow}$.
From these, we have a bijection between a Dyck tiling and a decreasing label through 
an Hermite history.

An example of an Hermite history is given in Figure \ref{fig:Hh}.
The corresponding tree $T$ is characterized by the bottom Dyck path $UDUDU^2D^2UD$,
and its decreasing label has the post-order word $51423$, and the pre-order 
word $54123$ read from right to left.

\begin{figure}[ht]
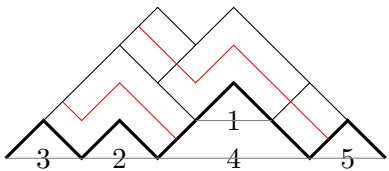

\tikzpic{-0.5}{[scale=0.5]
\draw[very thick](0,0)--(1,1)--(2,0)--(3,1)--(4,0)--(6,2)--(8,0)--(9,1)--(10,0);
\draw(1,1)--(4,4)--(5,3)--(6,4)--(9,1)(3,3)--(5,1)(4,2)--(5,3)(8,2)--(7,1);
\draw(9,0)node{$5$}(6,1)node{$1$}(6,0)node{$4$}(3,0)node{$2$}(1,0)node{$3$};
\draw[gray](0,0)--(10,0)(5,1)--(7,1);
\draw[red](3.5,3.5)--(5,2)--(6,3)--(8.5,0.5)(1.5,1.5)--(2,1)--(3,2)--(4.5,0.5);
}
\caption{An Hermite history and an decreasing label}
\label{fig:Hh}
\end{figure}

\subsection{Relation between 
\texorpdfstring{$L^{\uparrow}$}{Lup} and \texorpdfstring{$L^{\downarrow}$}{Ldown}}
Let $D$ be a Dyck tiling and $L^{\uparrow}$ the increasing label of $D$ through 
the DTS operation .
Similarly, denote by $L^{\downarrow}$ the decreasing label of $D$ through 
the Hermite history.

We denote the operation to exchange $i$ and $n+1-i$ by the bar operation, i.e., 
$\bar{i}=n+1-i$.
\begin{prop}
\label{prop:Lupdown}
Two labels $L^{\uparrow}$ and $L^{\downarrow}$ satisfy
\begin{align}
\label{eq:LupLd}
L^{\downarrow}(e)=\overline{L^{\uparrow}(e)},
\end{align}
where $e\in E_{n}(T)$.
\end{prop}
\begin{proof}
We prove the statement by induction on $n$. When $n=2$, we have only three Dyck tilings 
and it is easy to check Eq. (\ref{eq:LupLd}) by a calculation.
Suppose that Eq. (\ref{eq:LupLd}) holds true up to $n-1$.
Let $D$ be a Dyck tiling of size $n-1$, and $L^{\uparrow}$ and $L^{\downarrow}$ be 
the increasing and decreasing label for the tiling $D$.
We insert $n$ by the DTS operation into $L^{\uparrow}$.
On the other hand, we insert $1$ by the Hermite history into $L^{\downarrow}$ 
after increasing all the labels of $D$ by one.
We insert a peak $UD$ into the tiling $D$ for the both cases.
Note that the positions of the insertion are the same for both cases.
In the DTS operation, we have an addition of trivial Dyck tiles.
This process reflects that there are several labels 
which are right to $1$ in the Hermite history, and we have to add trivial Dyck tiles
to construct a Dyck tiling of size $n$ from $D$.
From these, it is easy to see that we have Eq. (\ref{eq:LupLd}).
\end{proof}

Below, we often identify the label $L^{\uparrow}$ with the label 
$L^{\downarrow}$ by Proposition \ref{prop:Lupdown}.

\subsection{Generating functions}
By the correspondence between $L^{\uparrow}$ and $L^{\downarrow}$ 
through the DTS operations and Hermite histories,
we have two descriptions of a Dyck tiling via $L^{\uparrow}$ and $L^{\downarrow}$.
Below, we mainly work on $L^{\downarrow}$ rather than $L^{\uparrow}$.

Consider a Dyck path $\lambda(n,m):=U^{n}D^{n}U^{m}D^{m}$.
Let $L^{\downarrow}_{0}$ be a decreasing label of the tree $T:=T(\lambda(n,m))$.
By construction of the tree $T$ for $\lambda(n,m)$, 
it is obvious that the label $L^{\downarrow}_{0}$ is $312$-avoiding. 
Consider the Hermite history for the label $L^{\downarrow}_{0}$ on $\lambda(n,m)$.
The Hermite history gives a unique Dyck tiling on $\lambda(n,m)$, and this Dyck 
tiling contains only trivial Dyck tiles.
We enumerate the up steps in $\lambda(n,m)$ by $1,2,\ldots, n+m$ from left to right.
We denote by $u_{i}$, $1\le i\le m$, the $n+i$-th up step from left.
By an Hermite history, we have $\mu_{i}$ unit boxes (which are trivial Dyck tiles)
associated to the up step $u_{i}$.
Since $L^{\downarrow}_{0}$ is decreasing, we have $\mu_{i}\ge \mu_{i+1}$ for 
$1\le i\le m-1$.

Define a Young diagram $\mu:=\mu(L^{\downarrow}_{0};\lambda(n,m))=(\mu_1,\mu_2,\ldots, \mu_{m})$.
Given two Young diagram $\nu=(\nu_1,\ldots,\nu_{m})$ and $\nu'=(\nu'_1,\ldots,\nu'_{m})$,
we write $\nu'\subseteq\nu$ if and only if $\nu'_{i}\le\nu_{i}$ for all $1\le i\le m$.
We define $|\nu|:=\sum_{i=1}^{m}\nu_{m}$, namely, $|\nu|$ is the number of unit boxes 
in $\nu$.

Then, we define the generating function
\begin{align}
\label{eq:defYmu}
\begin{aligned}
Y^{\uparrow}(\mu)&:=\sum_{\mu'\subseteq\mu}q^{|\mu'|}, \\
Y^{\downarrow}(\mu)&:=\sum_{\mu'\subseteq\mu}q^{|\mu|-|\mu'|}.
\end{aligned}
\end{align}

In general, the generating function $Y(\mu)$ has no factorization formula like 
Theorem \ref{thrm:FacDT1}.
For example, we have 
\begin{align*}
Y^{\downarrow}(\mu)=1+2q+q^2+q^3,
\end{align*}
if $\mu=(2,1)$.
However, when $\mu$ is a rectangle, i.e., $\mu=(p^{q})$ for some integer 
$p$ and $q\le m$, $Y(\mu)$ has a factorization formula 
\begin{align*}
Y^{\uparrow}(p^q)=Y^{\downarrow}(p^{q})=\genfrac{}{}{}{}{[p+q]!}{[p]![q]!}.
\end{align*}
This formula can be easily obtained by considering the number of up-right lattice paths 
in the rectangle $\mu$.

Let $D$ be a Dyck tiling. Then, we assign a weight $\mathrm{wt}(d)$ to a 
Dyck tile $d$ by use of the size $l(d)$ of the Dyck path characterizing $d$,
{\it i.e.},
\begin{align*}
\mathrm{wt}(d):=l(d)+1.
\end{align*}
By definition, the weight of a trivial Dyck tile is one.
Then, we define the weight $\mathrm{wt}(D)$ of a Dyck tiling
by the sum of the weights of Dyck tiles which form $D$.

\begin{defn}
Let $\lambda$ and $\mu$ be two Dyck paths such that $\mu$ is above $\lambda$.
The generating function of Dyck tilings is defined to be 
\begin{align*}
\mathtt{Dyck}(\lambda,\mu)
:=\sum_{D}\prod_{d\in D}q^{\mathrm{wt}(d)},
\end{align*}
where $D$ is a Dyck tiling in the region above $\lambda$ and below $\mu$, 
and $d$ is a Dyck tile in $D$.
\end{defn}

Given two Dyck paths $\lambda$ and $\mu$, we define the following generating 
function.
\begin{defn}
We define 
\begin{align*}
Z(\lambda,\mu):=\sum_{\lambda\le\nu\le\mu}\mathtt{Dyck}(\lambda,\nu).
\end{align*}
\end{defn}

\begin{example}
Consider $\lambda=(UD)^3$ and $\mu=U^3D^3$. Then, we have the six Dyck tilings above $\lambda$.
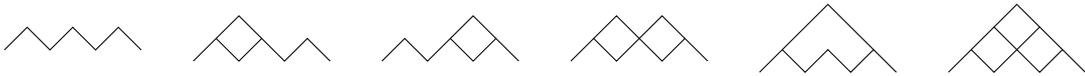
\begin{align*}
\tikzpic{-0.5}{[scale=0.3]
\draw(0,0)--(1,1)--(2,0)--(3,1)--(4,0)--(5,1)--(6,0);	
}\quad
\tikzpic{-0.5}{[scale=0.3]
\draw(0,0)--(1,1)--(2,0)--(3,1)--(4,0)--(5,1)--(6,0);
\draw(1,1)--(2,2)--(3,1);
}\quad
\tikzpic{-0.5}{[scale=0.3]
\draw(0,0)--(1,1)--(2,0)--(3,1)--(4,0)--(5,1)--(6,0);
\draw(3,1)--(4,2)--(5,1);
}\quad\tikzpic{-0.5}{[scale=0.3]
\draw(0,0)--(1,1)--(2,0)--(3,1)--(4,0)--(5,1)--(6,0);
\draw(1,1)--(2,2)--(3,1);
\draw(3,1)--(4,2)--(5,1);
}\quad
\tikzpic{-0.5}{[scale=0.3]
\draw(0,0)--(1,1)--(2,0)--(3,1)--(4,0)--(5,1)--(6,0);
\draw(1,1)--(3,3)--(5,1);
}\quad
\tikzpic{-0.5}{[scale=0.3]
\draw(0,0)--(1,1)--(2,0)--(3,1)--(4,0)--(5,1)--(6,0);
\draw(1,1)--(2,2)--(3,1);
\draw(3,1)--(4,2)--(5,1);\draw(2,2)--(3,3)--(4,2);
}
\end{align*}
The generating function is given by 
\begin{align*}
Z(\lambda,\mu)&=1+2q+2q^2+q^{3}, \\
&=[3][2].
\end{align*}
\end{example}

Let $T$ be a tree with $n$ edges.
Suppose that $L^{\uparrow}_{0}$ is an increasing label whose post-order word 
from right to left is $(n,n-1,\ldots,1)$.

Recall that Definition \ref{defn:ZLT} gives the generating function $Z(L,T)$ of 
labeled trees $L'$ with the tree $T$. 
\begin{theorem}[{\cite[Conjecture 1]{KW11}}, {\cite[Theorem 1.1]{K12}}]
\label{thrm:FacDT1}
The generating function of labeled trees on $T$ is given by 
\begin{align}
\label{eq:ZL0T}
Z(L^{\uparrow}_{0},T)	
=\genfrac{}{}{}{}{[n]!}{\prod_{a\in\mathrm{Arc}(\lambda)}[|a|]},
\end{align}
where $\lambda$ is the Dyck path corresponding to $T$, and $|a|$ is the length of a chord pair $a$.
\end{theorem}
\begin{proof}
The label $L^{\uparrow}_{0}$ is an increasing tree, and its post-order word from right to left
is $n\ldots21$. 
The inversion number of this post-order word is zero. 
The generating function involves all increasing trees $L'$ such that $L^{\uparrow}_{0}\le L'$.
By the DTS operation, each $L'$ corresponds to a Dyck tiling above the Dyck path $\lambda(T)$.
Further, it is clear that the inversion number $\mathrm{inv}(\omega(L'))$ is equal to 
the weight of the Dyck tiling.
Therefore, we have 
\begin{align*}
Z(L^{\uparrow}_{0},T)=Z(\lambda,\lambda_{\mathrm{top}}),
\end{align*}
where $\lambda_{\mathrm{top}}$ is the top path $U^nD^n$.
Then, $Z(\lambda,\lambda_{\mathrm{top}})$ has a factorized expression as in the r.h.s. 
of Eq. (\ref{eq:ZL0T}) by \cite{KW11,K12,KMPW12}.
\end{proof}

\begin{remark}
The expression (\ref{eq:ZL0T}) has the form of the tree hook-length formula 
by Knuth \cite[page 70]{Knu73} for the number of linear extensions of the poset on trees. 
Bj\"orner and Wachs obtained the $q$-analogue of Knuth's tree hook-length 
formula in \cite{BjoWac89}. Once we have a correspondence between a Dyck tiling and a labeled
tree, the expression (\ref{eq:ZL0T}) follows from the $q$-analogue of hook-length formula
and the DTS operation.
The proof of Theorem \ref{thrm:FacDT1} is based on the proof given in \cite{KMPW12}.
\end{remark}

The expression (\ref{eq:ZL0T}) implies that the generating function $Z(L_0^{\uparrow},T)$ 
is given by the fraction of $q$-integers.
In Section \ref{sec:gfLT}, we consider the generating function of labeled trees by relaxing 
the conditions on $L_0^{\uparrow}$. 
The generating function is no longer written as the hook-length formula in general,
however, it still has a factorization. Each factor in the expression comes from the 
combinatorial structure of a tree.

\section{Determinant expressions of trivial Dyck tilings}
\label{sec:LGV}
We say that a Dyck tiling above $\lambda$ is a trivial Dyck tiling if it contains only trivial Dyck 
tiles, {\it i.e.}, it does not contain non-trivial Dyck tiles.
Then, it is easy to see that a trivial Dyck tiling is bijective to a Dyck path 
above the Dyck path $\lambda$.

We first recall the Lindstr\"om--Gessel--Viennot lemma \cite{GesVie85,Lin73} 
which gives an expression of the number of lattice paths in terms of a determinant.
Let $X:=\{x_1,\ldots,x_n\}$ and $Y:=\{y_1,\ldots,y_n\}$ be the set of 
vertices. 
Given two vertices $x$ and $y$, we denote by $w(x,y)$ the sum of the 
weights given to paths from $x$ to $y$.
If we give the weight one to a path from $x$ to $y$, the sum $w(x,y)$
is nothing but the number of paths from $x$ to $y$.
An $n$-tuple of non-intersecting paths from $X$ to $Y$ is an $n$-tuple 
$(P_1,\ldots,P_{n})$ satisfying the following two conditions: 
(1) there exists a permutation $\rho$ of $[1,n]$ such that 
$P_i$ is a paths from $x_i$ to $y_{\rho(i)}$,
and (2) if $i\neq j$, the paths $P_{i}$ and $P_{j}$ do not occupy the same 
vertex.
Let $\mathcal{P}:=(P_1,\ldots,P_n)$ and $\rho(P)$ be the permutation satisfying 
the above conditions (1) and (2). 
We denote by $w(P_{i}):=w(x_i,y_{\rho(i)})$ the weight given to the path 
$P_{i}$ from $x_i$ to $y_{\rho(i)}$.
\begin{lemma}[\cite{GesVie85,Lin73}]
The weighted sum of non-intersecting paths from $X$ to $Y$ is given by 
the following determinant: 
\begin{align*}
\det[w(x_i,y_{j})]_{1\le i,j\le n}=
\sum_{\mathcal{P}:X\rightarrow Y}\mathrm{sign}(\rho(P))\prod_{i=1}^{n}w(P_{i}).
\end{align*}
\end{lemma}

We apply the Lindstr\"om--Gessel--Viennot lemma to an enumeration of trivial Dyck tilings 
above $\lambda(n,m):=U^nD^nU^mD^m$.
Let $\mu:=(\mu_1,\ldots, \mu_{m})$ be a Young diagram above $\lambda(n,m)$.
We have simple conditions $n\ge\mu_1\ge\mu_2\ldots\ge\mu_{m}\ge0$.
We define $2m$ points $\{a_i,b_i: 1\le i\le m\}$ on $\mathbb{Z}^{2}$ as follows:
\begin{align*}
a_{i}&=(i-1,i-1), \\
b_{i}&=a_{i}-(0,1)+(\mu_{m+1-i},0)=(\mu_{m+1-i}+i-1,i-2).
\end{align*}
In $\mathbb{Z}^2$, we consider right-down paths from $a_{i}$ to $b_{j}$. 
We denote by $\mathcal{P}(a_i,b_{j})$ the set of such paths from $a_{i}$ to $b_{j}$.
Let $p_{i,j}\in\mathcal{P}(a_i,b_{j})$ be a path.
By definition, a path $p_{i,j}$ consists of right steps and down steps.
Let $r$ be a right step in $p_{i,j}$ and $y(r)$ be the $y$-coordinate of $r$ in $\mathbb{Z}^2$.
Then, we define the weight of $p_{i,j}$ by 
\begin{align*}
\mathrm{wt}(p_{i,j}):=\prod_{r\in p_{i,j}}q^{y(r)+1}.
\end{align*}
By using the weight for a path $p_{i,j}$, we define a weight for all the paths from $a_{i}$ to $b_{j}$
by
\begin{align*}
\mathrm{wt}(a_{i}\rightarrow b_{j}):=\sum_{p_{i,j}\in\mathcal{P}(a_i,b_j)}\mathrm{wt}(p_{i,j}).
\end{align*}
Then, we define the determinant by 
\begin{align}
\label{eq:wtab}
\mathrm{Det}(\mathbf{a}\rightarrow\mathbf{b}):=
q^{-D}\det\left(\mathrm{wt}(a_{i}\rightarrow b_{j})\right)_{1\le i,j\le n},
\end{align}
where 
\begin{align*}
D=\sum_{i=1}^{m}(i-1)\mu_{m+1-i}.
\end{align*}

Recall that $Y^{\uparrow}(\mu)$ in Eq. (\ref{eq:defYmu}) is the generating function of Young diagram inside $\mu$.
The next proposition is a direct consequence of the Lindstr\"om--Gessel--Viennot formula.
\begin{prop}
\label{prop:Ydet}
The generating function $Y^{\uparrow}(\mu)$ has a determinant expression:
\begin{align}
\label{eq:Ydet1}
Y^{\uparrow}(\mu)=\mathrm{Det}(\mathbf{a}\rightarrow\mathbf{b}).
\end{align}
\end{prop}

To generalize Eq. (\ref{eq:Ydet1}) to a Dyck tiling above $\lambda$, we 
first study Dyck tilings for $312$-avoiding decreasing labels.
Note that the choice of a decreasing label $L^{\downarrow}$ (which may not 
be $312$-avoiding) and $T$ uniquely fixes a Dyck tiling $D$. 
This Dyck tiling $D$ may contain a non-trivial Dyck tile in general.
However, the pattern avoidance in $L^{\downarrow}$ imposes a 
strong condition on a Dyck tiling as in the next proposition.

\begin{prop}
\label{prop:312trivial}
Suppose $L^{\downarrow}$ is a $312$-avoiding label of $T$. Then, the 
Dyck tiling $D$ characterized by $L^{\downarrow}$ and $T$ consists of 
only trivial Dyck tiles. Equivalently, $D$ contains no Dyck tiles of 
size $m>1$.
\end{prop}
Before proceeding to the proof of Proposition \ref{prop:312trivial}, 
we give a remark.
\begin{remark}
The converse of the proposition is not  true. In some cases, a Dyck tiling 
characterized by $L^{\downarrow}$ and $T$ contains only trivial Dyck tiles 
even if $L^{\downarrow}$ is not $312$-avoiding.
For example, take $T=U^2D^2U^2D^2UD$ and $\omega(L^{\downarrow})=52413$ ($\omega$ is the 
pre-order word read from left to right).
We have the following Dyck tiling:
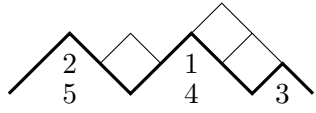
\begin{align*}
\tikzpic{-0.5}{[scale=0.4]
\draw[very thick](0,0)--(2,2)--(4,0)--(6,2)--(8,0)--(9,1)--(10,0);
\draw(3,1)--(4,2)--(5,1)(6,2)--(7,3)--(9,1)(7,1)--(8,2);
\draw(2,1)node{$2$}(2,0)node{$5$};\draw(6,1)node{$1$}(6,0)node{$4$};
\draw(9,0)node{$3$};
}
\end{align*}
The label $L^{\downarrow}$ contains a $312$-pattern, but its Dyck tiling does not 
contain non-trivial Dyck tiles.
\end{remark}

\begin{proof}[Proof of Proposition \ref{prop:312trivial}]
We first rephrase the statement in terms of $L^{\uparrow}$.
The label $L^{\downarrow}$ is $312$-avoiding is equivalent to the property 
that $L^{\uparrow}$ is $132$-avoiding.
Then, by the correspondence between $L^{\uparrow}$ and $L^{\downarrow}$ through 
the DTS operation and the Hermite history, it is enough to show that 
the Dyck tiling $D$ characterized by $L^{\uparrow}$ and $T$ does not contain non-trivial
Dyck tiles.
However, this is obvious from the definition of the DTS operation. Namely, 
if $L^{\uparrow}$ is $132$-avoiding, we do not have a non-trivial Dyck tile of size $m\ge1$.
This completes the proof.
\end{proof}

Let $\lambda$ be a Dyck path and fix a Dyck tiling $D$ above $\lambda$ consisting of only trivial 
Dyck tiles.
We denote by $L^{\uparrow}$ the increasing label corresponding to the Dyck tiling $D$ 
through the DTS operation.
Let $\mu:=(\mu_1,\ldots,\mu_n)$ be the number of trivial Dyck tiles in $D$ which 
are associated to the $i$-th up steps from top to bottom.
Let $u_{i}$ be the $i$-th up steps from the top in $\lambda$, and 
$x(u_{i})$ be its $x$-coordinate in $\mathbb{Z}^2$.
By definition of a Dyck path, we have $x(u_{i})\ge x(u_{i+1})$ for $1\le i\le n-1$.
We define $2n$ points $\{c_i,d_{i}: 1\le i\le n\}$ in $\mathbb{Z}^{2}$ 
as follows:
\begin{align*}
c_{1}&=(0,0), \\
c_{i+1}&=c_{i}+(1,1)+(x(u_{i})-x(u_{i+1}),0), \quad \forall i\in[1,n-1],\\
d_{i}&=c_{i}-(0,1)+(\mu_{i},0), \quad i\in[1,n].
\end{align*}
We consider right-down lattice paths from $c_{i}$ to $d_{j}$.
Then, we define the determinant $\mathrm{Det}(\mathbf{c}\rightarrow\mathbf{d})$
as in Eq. (\ref{eq:wtab}) by replacing $a_{i}$ and $b_j$ by $c_i$ and $d_{j}$ 
respectively. 
We define the determinant associated to $L^{\uparrow}$ and $T$ by 
\begin{align}
Y(L^{\uparrow},T):=\mathrm{Det}(\mathbf{c}\rightarrow\mathbf{d}).
\end{align}

Again by the Lindstr\"om--Gessel--Viennot formula, we have the following proposition.

\begin{prop}
\label{prop:YDyck}
Suppose that a Dyck tiling $D$ contains only trivial Dyck tiles, and 
$L^{\uparrow}$ is an increasing label for $D$.
The generating function $Y(L^{\uparrow},T)$ is nothing but 
the generating function for Dyck paths above $\lambda$ and below $\lambda'$ 
where $\lambda'$ is a top path of the Dyck tiling $D$.
\end{prop}

\section{Generating functions for labeled trees}
\label{sec:gfLT}
\subsection{Recursive formula}
In this subsection, we derive a recursive formula for a generating 
function $Z(L^{\downarrow},T)$.

Let $L^{\downarrow}$ be a decreasing label of the tree $T$.
Let $l_1$ be the edge with label $1$ and $l_2,\ldots,l_{m}$ be 
the leaves of the tree $T$ which are right to $l_1$.
We assume that $l_{i+1}$ is right to $l_{i}$.
Note that $l_1$ is a leaf in $T$ since we consider a decreasing label.
Denote by $T_{i}$ the tree obtained from $T$ by deleting the leaf $l_{i}$
for $1\le i\le m$.
Let $L_{i}^{\downarrow}$ be the label of $T$ such that 
\begin{enumerate}
\item the leaf $l_{i}$ has the label $1$, and
\item the label $L^{\downarrow}_{i}$ is characterized by the unrefinable 
sequence of labels 
\begin{align*}
L^{\downarrow}\lessdot L_1\lessdot\ldots\lessdot L_{p}=L^{\downarrow}_{i},
\end{align*}
where $p$ is minimal. We write $p(L_{i}^{\downarrow}):=p$.
\end{enumerate}
We define $\widetilde{L^{\downarrow}_{i}}$ as the label of $T_{i}$ obtained 
from $L_{i}^{\downarrow}$ by deleting the leaf $l_{i}$ and 
all the labels are decreased by one.
The definition of $\widetilde{L^{\downarrow}_{i}}$ is well-defined since 
we consider a decreasing label and $L_{i}^{\downarrow}$ has the integer $1$
on the leaf $l_i$.

\begin{theorem}
Let $\widetilde{L^{\downarrow}_{i}}, 1\le i\le m$, be labels as above.
Then, we have
\begin{align}
\label{eq:Zsum}
Z(L^{\downarrow},T)=\sum_{i=1}^{m}q^{p(L_{i}^{\downarrow})}\cdot Z(\widetilde{L_{i}^{\downarrow}},T_{i}).
\end{align}
\end{theorem}
\begin{proof}
The generating function $Z(L^{\downarrow},T)$ 
counts the number of labels $L$ such that $L^{\downarrow}\le L$ with 
the weight given by $q^{\mathrm{inv}(\omega(L))}$.
By definition of the covering relation, $L$ has a label $1$ on one of 
the leaves $l_1,\ldots,l_m$.
We cannot move the label $1$ left to the leaf $l_1$ by the cover relation in Definition \ref{def:coverL}.
Thus, we have a natural bijection between $L$ and $\widetilde{L}$ such 
that $L^{\downarrow}\le L$ and $\widetilde{L_{i}^{\downarrow}}\le \widetilde{L}$
with $1\le i\le m$.
Recall that we delete the edge $l_i$ from $L_{i}^{\downarrow}$ to 
obtain $\widetilde{L_{i}^{\downarrow}}\le \widetilde{L}$.
This means that the difference of the weights between $L$ and $\widetilde{L_{i}^{\downarrow}}\le \widetilde{L}$
is given by $p(L_{i}^{\downarrow})$.
From these, we have Eq. (\ref{eq:Zsum}).
\end{proof}

Let $L^{\downarrow}$ be a decreasing label and $\widetilde{L_{i}^{\downarrow}}$, $L^{\downarrow}_{i}$, 
$1\le i\le m$, be also the decreasing labels obtained from $L^{\downarrow}$ as above.
\begin{prop}
Let $L^{\downarrow}$ be a $312$-avoiding label.
Then, the labels $\widetilde{L_{i}^{\downarrow}}$, $1\le i\le m$, are 
$312$-avoiding.
\end{prop}
\begin{proof}
The leaf $l_1$ has the label $1$ by definition of $L^{\downarrow}$, and 
the labels of the leaves $l_{i}, 2\le i\le m$ satisfy $L^{\downarrow}(l_i)\ge2$.
We first assume that 
\begin{enumerate}
\item[($*$)] 
$L^{\downarrow}(l_{i'})>L^{\downarrow}(l_{i})$ for $2\le i'\le i-1$. 
\end{enumerate}
When we construct $L^{\downarrow}_{i}$ from $L^{\downarrow}$, we exchange 
the labels of $l_1$ and $l_{i}$, namely, $L^{\downarrow}_{i}(l_{1})=L^{\downarrow}(l_{i})$
and $L^{\downarrow}_{i}(l_{i})=L^{\downarrow}(l_{1})=1$.
This exchange of the labels may produce a label $L'_{i}$  which is in general neither increasing 
nor decreasing. We construct the decreasing label $L^{\downarrow}_{i}$ from $L'_{i}$ by reordering 
the integers in the following way.

Consider the sequence of edges $p(l_1)$ from $l_1$ to the root in $T$. 
Let $l'_1$ be the parent edge of the leaf edge $l_1$. If the label of $l'_1$ is larger
than that of $l_1$, then we have a decreasing label. In this case, we are done.
Suppose that the label of $l'_1$ is smaller than that of $l_1$.
We relabel the edges in $p(l_1)$ such that the labels are decreasing from the root 
to the leaf $l_1$.
Then, it is obvious that the new label is admissible as a deceasing label.
In this way, we obtain the decreasing label $L^{\downarrow}_{i}$ from $L'_{i}$.

Recall that $L^{\downarrow}$ is $312$-avoiding and we delete $l_i$ (whose label is one)
from $L_{i}^{\downarrow}$ to obtain $\widetilde{L_{i}^{\downarrow}}$.
The label $L^{\downarrow}_{i}$ may not be $312$-avoiding. 
However, we claim that the above construction 
of $L'_{i}$ implies that a $312$-pattern in $L_{i}^{\downarrow}$ has to contain 
the label $1$ on the edge $l_i$.
Note that the difference of the label between $L^{\downarrow}$ and $L^{\downarrow}_{i}$
is the labels of the edges in $p(l_1)$ and that of $l_i$.
Since $L^{\downarrow}$ is $312$-avoiding, if $L^{\downarrow}_{i}$ has a pattern $312$,
then this pattern contains the labels of the edges in $p(l_1)$ or that of $l_i$.
Recall that we reorder the labels of the edges in $p(l_1)$ to obtain $L^{\downarrow}_{i}$.
This operation moves the label downward at most one edge.
Therefore, if $L^{\downarrow}_{i}$ has a $312$-pattern, this pattern contains the label of the 
edge $l_1$ or $l_{i}$. 
First, suppose that $L^{\downarrow}_{i}$ has a pattern $312$ which contains
the label $l_1$. 
We consider the following three cases:
\begin{enumerate}
\item
Let $e_{l}$, $e_{m}$ be the edges such that $e_{l}$ is strictly left to $e_{m}$, and $e_{m}$
is strictly left to $l_1$.
Suppose we have a $312$-pattern on these three edges, {\it i.e.}, 
$L_{i}^{\downarrow}(e_{m})<L_{i}^{\downarrow}(e_{l})<L_{i}^{\downarrow}(l_1)$.
since $e_{l}, e_{m}\notin p(l_1)$, the labels on these edges in $L^{\downarrow}$ are 
the same as $L_{i}^{\downarrow}$. Therefore, this means that we have a $312$-pattern 
which consists of $e_{l}$, $e_{m}$ and $l_i$ in $L^{\downarrow}$. However, this contradicts 
the fact that $L^{\downarrow}$ is $312$-avoiding. 
\item  Let $e_{l}$, $e_{r}$ be the edges such that $e_{l}$ is strictly left to $l_1$, and 
$e_{r}$ is strictly right to $l_1$.
Suppose we have a $312$-pattern on these three edges in $L^{\downarrow}_{i}$. 
The label of $e_r$ is larger than that of $l_1$ in $L^{\downarrow}$, which implies that 
we have a $312$-pattern in $L^{\downarrow}$. This is a contradiction.
\item 
Let $e_{m}$, $e_{r}$ be the edges such that $e_{m}$ is strictly right to $l_1$, and $e_{r}$ is
strictly right to $e_{m}$. Suppose we have a $312$-pattern on these three edges $l_1$, $e_{m}$
and $e_{r}$.  
By the condition ($*$), $e_{m}$ is right to $l_{i}$ since the label of $e_{m}$ is smaller than that
of $l_1$ in $L^{\downarrow}_{i}$. This means that we have a $312$-pattern on the edges 
$l_{i}$, $e_{m}$, and $e_{r}$ in $L^{\downarrow}$. This is a contradiction.
\end{enumerate}
Secondly, we consider the case where $L^{\downarrow}_{i}$ has a pattern $312$ which contains 
the label $l_{i}$.
Since we delete $l_i$ to obtain $\widetilde{L^{\downarrow}_{i}}$, this $312$-pattern 
disappears  in $\widetilde{L^{\downarrow}_{i}}$.
Together with the fact that $L^{\downarrow}$ is $312$-avoiding, 
the label $\widetilde{L_{i}^{\downarrow}}$ is also $312$-avoiding.

Instead of ($*$), we consider the case where there exists $i'_1<i'_{2}<\ldots<i'_{k}$ 
such that 
$1<L^{\downarrow}(l_{i'_{1}})<L^{\downarrow}(l_{i'_{2}})<\ldots<L^{\downarrow}(l_{i'_{k}})<L^{\downarrow}(l_{i})$.
Then, we first construct $L^{\downarrow}_{i'_1}$ as in the first case.
Then, we construct $L^{\downarrow}_{i'_2}$ from $L^{\downarrow}_{i'_1}$.
More in general, we construct $L^{\downarrow}_{i'_{p+1}}$ from $L^{\downarrow}_{i'_p}$ by 
exchanging the labels of $l_{i'_{p}}$ and $l_{i'_{p+1}}$ in $L^{\downarrow}_{i'_p}$, 
and by reordering the labels of edges to have a decreasing label.
By a similar argument to the first case, if $L_{i}^{\downarrow}$ contains a $312$-pattern,
then this $312$-pattern has to contain the label $1$ on the edge $l_{i}$.
Again by the same argument as the first case, $\widetilde{L_{i}^{\downarrow}}$
is $312$-avoiding, which completes the proof.
\end{proof}

Let $n$ be a branch point in $T$ and $L$ be a label on $T$.
We transform $T$ into another tree $T'$ by transforming the 
subtree below $n$ into a subtree such that its root is $n$ and 
it has a unique leaf.
Since $L$ is an increasing or decreasing label, we have a 
unique label $L'$ which is increasing or decreasing.
Note that the labels on the edges which are not below $n$ 
remain the same before and after the transformation of the tree.

\begin{lemma}
\label{lemma:312avoiding}
Let $L$ and $L'$ be labels on the tree $T$ and $T'$ respectively.
Then, if $L$ is $312$-avoiding, then $L'$ is also $312$-avoiding.
\end{lemma}
\begin{proof}
Transformation of $T$ into $T'$ reduces the number of leaves.
Further, the labels on the edges below $n$ in $T'$ are decreasing 
from $n$ to the unique leaf.
Then, it is easy to see that $L$ is $312$-avoiding implies 
$L'$ is also $312$-avoiding.
\end{proof}

Before proceeding to the main theorem, we introduce a condition on a decreasing label
which plays a central role in its proof.
We consider the following condition on a decreasing label $L^{\downarrow}$:
\begin{enumerate}
\item[($**$)]
Let $e$ and $e'$ be the edges of $T$ such that they have the common parent node (possibly the root)
and $e$ is right to $e'$.
Then, we consider the following condition on the labels of $e$ and $e'$:
\begin{align*}
L^{\downarrow}(e')<L^{\downarrow}(e).
\end{align*}
\end{enumerate}

Let $\lambda$ be a Dyck path of size $n$. We divide the region above $\lambda$ and below 
$U^nD^n$ into several rectangles.
Take a branch point $b$ of a tree such that the subtree $T$, whose root is $b$, has no branch 
point except its root.
This means that the Dyck path $\lambda'$ corresponding to $T$ is written as 
a concatenation of smaller Dyck paths $\lambda'=\lambda'_1\circ\lambda'_2\circ\ldots\circ\lambda'_{k}$. 
Here, $\lambda'_{i}$ is written as $U^{m_{i}}D^{m_{i}}$ for some positive $m_{i}$.
We construct $k-1$ Dyck paths $\nu_j$, $1\le j\le k-1$, above $\lambda'$ as follows.
The Dyck path $\nu_{j}$ is written as 
\begin{align*}
U^{m_1+\ldots+m_{j+1}}D^{m_1+\ldots+m_{j+1}}\circ\lambda'_{j+2}\circ\ldots\circ\lambda'_{k}.
\end{align*}
Note that $\nu_{k}$ is written as $U^{N}D^{N}$ with $N=\sum_{j=1}^{k}m_{j}$.
We denote by $R_{j+1}$ the region above $\nu_{j}$ and below $\nu_{j+1}$ with $\nu_0=\lambda'$. 
By its construction, the region $R_{j}$ is a rectangle.

We replace $\lambda'$ by $\nu_{k}$ in $\lambda$. The number of branch points are decreased by one.
We perform this process on the new Dyck path until we visit all the branch points. 
It is clear that we have several rectangles above $\lambda$ and below $U^nD^n$.
\begin{figure}[ht]
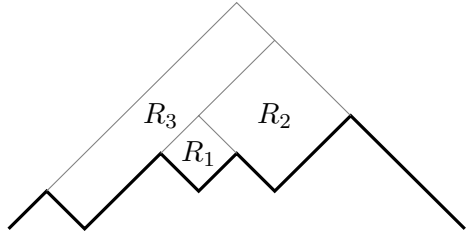

\begin{align*}
\tikzpic{-0.5}{[scale=0.5]
\draw[very thick](0,0)--(1,1)--(2,0)--(4,2)--(5,1)--(6,2)--(7,1)--(9,3)--(12,0);
\draw[gray](4,2)--(5,3)--(6,2)(1,1)--(6,6)--(9,3)(5,3)--(7,5);
\draw(5,2)node{$R_1$}(7,3)node{$R_2$}(4,3)node{$R_3$};
}
\end{align*}
\caption{Division of the region above a Dyck path into rectangles}
\label{fig:defR}
\end{figure}
Figure \ref{fig:defR} is an example of the division of the region 
above $UDU^2DUDU^2D^3$ and below $U^6D^6$ into smaller rectangles 
$R_1$, $R_2$ and $R_{3}$.

Let $R_{i}$, $1\le i\le p$, be the rectangles above the Dyck path $\lambda$ obtained as above.
Recall that a decreasing label gives a Dyck tiling $D$. Suppose that the tiling $D$ has 
no non-trivial Dyck tiles. 
Then, we have a Young diagram $\mu_{i}$, $1\le i\le p$, in the region $R_{i}$.
The generating function can be factorized as follows.

\begin{theorem}
\label{thrm:GFLT312}
Let $L^{\downarrow}_{0}$ be a $312$-avoiding label which satisfies the condition ($**$).
Let $\mu$ be the top path of the Dyck tiling given by $L^{\downarrow}_{0}$.
Then, we have  
\begin{align}
\label{eq:ZinY}
Z(L^{\downarrow}_{0},T)=Z(\lambda,\mu)=\prod_{i=1}^{p}Y^{\downarrow}(\mu_{i}),
\end{align}
where $Y^{\downarrow}(\mu)$ is defined in Eq. (\ref{eq:defYmu}).
\end{theorem}
\begin{proof}
Since $L^{\downarrow}_{0}$ is $312$-avoiding, its subtrees 
are also $312$-avoiding.
As in Lemma \ref{lemma:312avoiding}, we transform the tree $T$ into 
$T'$ at the left-most and down-most branch point $n$ of $T$.
Then, by Lemma \ref{lemma:312avoiding}, we have a label $L'$ on 
the tree $T'$.
We compare the labels of $T$ and $T'$ below the branch point $n$. 
By this transformation, we have several rectangles $R_1,\ldots,R_{p}$
in the region surrounded by the Dyck paths corresponding to $T$ and $T'$.

Suppose that the branch point $n$ has $m$ child edges, and 
the these edges have integers $i_1,\ldots,i_{m}$. 
Then, the condition ($**$) implies that the root of the $m$ child edges 
can have any permutation of $i_1,\ldots, i_{m}$. 
Then, since the edges below $n$ in $T'$ have only a single leaf, 
the enumeration of the labeled trees in $T$ gives the generating function 
$\prod_{i=1}^{p}Y^{\downarrow}(\mu_i)$ where
$\mu_i$ is the Young diagram in the region $R_i$, $1\le i\le p$.

Then, we transform $T'$ into $T''$ at the left-most and down-most 
branch point $n'$ of $T'$.
By continuing this process, we have a sequence of trees 
\begin{align*}
T\rightarrow T'\rightarrow T''\rightarrow\ldots\rightarrow T^{(q)},
\end{align*}
where $T^{(q)}$ is a tree consisting of $n$ edges with a single leaf.
In each step, transformation gives the product of a generating function $Y^{\downarrow}(\mu_{i})$.
Then, the generating function is given by the product of generating functions
$Y^{\downarrow}(\mu_{i})$, $1\le i\le q'$ with some $q'$, which implies Eq. (\ref{eq:ZinY}).
\end{proof}

\begin{example}
Consider the following Dyck tiling corresponding to a decreasing label:
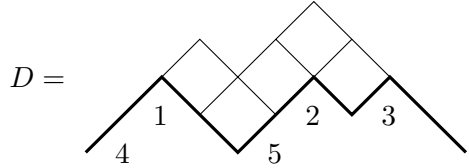
\begin{align*}
D=
\tikzpic{-0.5}{[scale=0.5]
\draw[very thick](0,0)--(2,2)--(4,0)--(6,2)--(7,1)--(8,2)--(10,0);
\draw(2,2)--(3,3)--(5,1)(3,1)--(6,4)--(8,2)(5,3)--(6,2)--(7,3);
\draw(0.5,0.5)node[anchor=north west]{$4$}(1.5,1.5)node[anchor=north west]{$1$}
(4.5,0.5)node[anchor=north west]{$5$}(5.5,1.5)node[anchor=north west]{$2$}
(7.5,1.5)node[anchor=north west]{$3$};
}
\end{align*}
The tiling $D$ satisfies the condition ($**$). 
The generating function is given by the product 
\begin{align*}
Z(D)=Y^{\downarrow}(\mu_1)Y^{\downarrow}(\mu_2)=(1+2q+2q^2+q^3+q^4)(1+q),
\end{align*}
where $\mu_1=(2,1,1)$ and $\mu_2=(1)$.
\end{example}

\begin{example}
Consider the following Dyck tiling corresponding to a decreasing label:
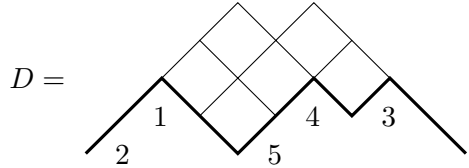
\begin{align*}
D=
\tikzpic{-0.5}{[scale=0.5]
\draw[very thick](0,0)--(2,2)--(4,0)--(6,2)--(7,1)--(8,2)--(10,0);
\draw(2,2)--(3,3)--(5,1)(3,1)--(6,4)--(8,2)(5,3)--(6,2)--(7,3);
\draw(3,3)--(4,4)--(5,3);
\draw(0.5,0.5)node[anchor=north west]{$2$}(1.5,1.5)node[anchor=north west]{$1$}
(4.5,0.5)node[anchor=north west]{$5$}(5.5,1.5)node[anchor=north west]{$4$}
(7.5,1.5)node[anchor=north west]{$3$};
}
\end{align*}
By a simple calculation, the generating function $Z(D)$ is given by
\begin{align*}
Z(D)=1+2q+4q^2+4q^3+3q^4+2q^5+q^6,
\end{align*}
which is not factorized any further since $D$ violates the condition ($**$).
\end{example}

Let $\lambda_{i}$ be the Dyck path of the form $\lambda_{i}=U^{m_i}D^{m_i}$ 
for some positive integer $m_i$.
Below, we consider a tree $T$ corresponding to the Dyck path
$\lambda=\lambda_1\circ\lambda_2\circ\ldots\circ\lambda_{p}$.
In other words, a tree $T$ has a unique branch point at the root of $T$.
In this case, we can relax the condition ($**$) on labels.

We divide the region above $\lambda$ and the top path into several rectangles 
$R_{1},\ldots, R_{p}$ as above (see the paragraphs above Theorem \ref{thrm:GFLT312}). 

\begin{prop}
Suppose that $\lambda$ is a concatenation of Dyck paths of the form $U^{m}D^{m}$ with $m\ge1$,  
$T$ is a tree corresponding to $\lambda$, and $L^{\downarrow}_{0}$ is 
a $312$-avoiding decreasing label on $T$.
Then, let $\mu_i$, $1\le i\le p$, be the Young diagram in the region $R_{i}$.
The generating function of decreasing labeled trees on $T$ is given by
\begin{align}
\label{eq:ZLY}
Z(L^{\downarrow}_{0},T)=\prod_{i=1}^{p}Y^{\downarrow}(\mu_i).
\end{align}
\end{prop}
\begin{proof}
Suppose that the region $R_{i}$ is left to the region $R_{i+1}$ for $1\le i\le p-1$.
Let $\lambda_{t}$ and $\lambda_{b}$ be the top path and bottom path which surround 
the region $R_{p}$. This means that $\lambda_{t}=U^{m_1+m_2}D^{m_1+m_2}$ and 
$\lambda_{b}=U^{m_1}D^{m_1}U^{m_2}D^{m_2}$.
Suppose that $\nu\subseteq\mu_{p}$ be the Young diagram inside $\mu_p$ in the region $R_{p}$.
The diagram $\nu$ in $R_{p}$ gives the decreasing label $L^{\downarrow}$ on $\lambda_{b}$
such that $L^{\downarrow}\le L^{\downarrow}_{0}$.
Next, we consider the region $R_{p-1}$ in a similar way. 
We fix a decreasing label on the Dyck path $\lambda_{i}$ one-by-one starting from $\lambda_{p}$.
Note that the obtained deceasing label $L$ satisfies $L^{\downarrow}_{0}\le L$.
It is clear that the number of decreasing labels is equal to the product of $Y^{\downarrow}(\mu_{i})$,
which implies Eq. (\ref{eq:ZLY}).
\end{proof}

\begin{remark}
The condition that a decreasing label is $312$ is unavoidable. 
For example, the following Dyck tiling consists of only trivial Dyck tiles, but 
it contains the $312$-pattern.
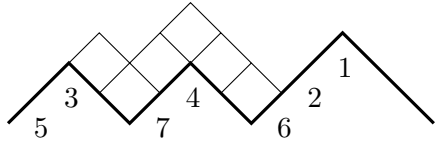
\begin{align*}
\tikzpic{-0.5}{[scale=0.4]
\draw[very thick](0,0)--(2,2)--(4,0)--(6,2)--(8,0)--(11,3)--(14,0);
\draw(2,2)--(3,3)--(5,1)(3,1)--(6,4)--(9,1)(5,3)--(6,2)--(7,3)(7,1)--(8,2);
\draw(0.5,0.5)node[anchor=north west]{$5$}(1.5,1.5)node[anchor=north west]{$3$}
(4.5,0.5)node[anchor=north west]{$7$}(5.5,1.5)node[anchor=north west]{$4$}
(8.5,0.5)node[anchor=north west]{$6$}(9.5,1.5)node[anchor=north west]{$2$}
(10.5,2.5)node[anchor=north west]{$1$};
}
\end{align*}
The generating function is given by 
\begin{align*}
(1+q)(1+3q+2q^2+3q^3+q^4+q^5).
\end{align*}
This generating function cannot be factorized any further.
\end{remark}

\begin{example}
\label{ex:D1D2}
Consider the two decreasing labels depicted in Figure \ref{fig:decL}.
Although these two Dyck tilings are a mirror image of another,  
the generating functions are different.  
\begin{figure}[ht]
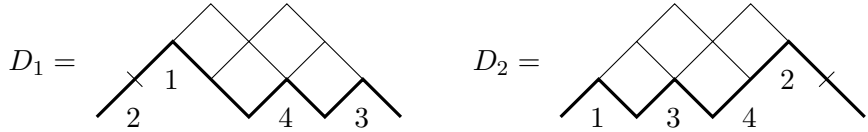

\begin{align*}
D_1=
\tikzpic{-0.5}{[scale=0.5]
\draw[very thick](0,0)--(2,2)--(4,0)--(5,1)--(6,0)--(7,1)--(8,0);
\draw(2,2)--(3,3)--(5,1)--(6,2)(3,1)--(5,3)--(7,1);
\draw(0.5,0.5)node[anchor=north west]{$2$}(1.5,1.5)node[anchor=north west]{$1$}
(4.5,0.5)node[anchor=north west]{$4$}(6.5,0.5)node[anchor=north west]{$3$};
\draw(0.8,1.2)--(1,1)--(1.2,0.8);
}\qquad
D_2=\tikzpic{-0.5}{[scale=0.5]
\draw[very thick](0,0)--(1,1)--(2,0)--(3,1)--(4,0)--(6,2)--(8,0);
\draw(1,1)--(3,3)--(5,1)(2,2)--(3,1)--(5,3)--(6,2);
\draw(6.8,0.8)--(7.2,1.2);
\draw(0.5,0.5)node[anchor=north west]{$1$}(2.5,0.5)node[anchor=north west]{$3$}
(4.5,0.5)node[anchor=north west]{$4$}(5.5,1.5)node[anchor=north west]{$2$};
}
\end{align*}
\caption{Two decreasing labels}
\label{fig:decL}
\end{figure}
The generating functions $Z(D_i)$ for the tiling $D_i$, $1\le i\le 2$, are given 
by 
\begin{align*}
Z(D_1)&=(1+q+q^2)^2, \\
Z(D_2)&=(1+q)(1+2q+q^2+q^3).
\end{align*} 
\end{example}

\begin{example}
\label{ex:D5}
Consider the following Dyck tiling associated to a permutation $12534$.
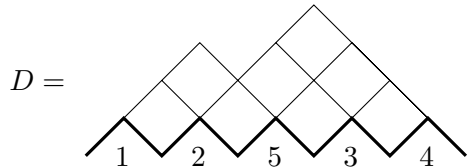
\begin{align*}
D=
\tikzpic{-0.5}{[scale=0.5]
\draw[very thick](0,0)--(1,1)--(2,0)--(3,1)--(4,0)--(5,1)--(6,0)--(7,1)--(8,0)--(9,1)--(10,0);
\draw(1,1)--(3,3)--(5,1)--(7,3)--(9,1)(2,2)--(3,1)--(6,4)--(9,1)(5,3)--(7,1)--(8,2);
\draw(0.5,0.5)node[anchor=north west]{$1$}(2.5,0.5)node[anchor=north west]{$2$}
(4.5,0.5)node[anchor=north west]{$5$}(6.5,0.5)node[anchor=north west]{$3$}
(8.5,0.5)node[anchor=north west]{$4$};
}
\end{align*}
The rectangle regions above $(UD)^{5}$ for the permutation $12534$ are depicted as 
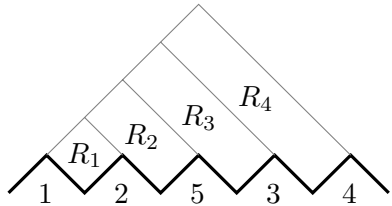
\begin{align*}
\tikzpic{-0.5}{[scale=0.5]
\draw[very thick](0,0)--(1,1)--(2,0)--(3,1)--(4,0)--(5,1)--(6,0)--(7,1)--(8,0)--(9,1)--(10,0);
\draw[gray](1,1)--(5,5)--(9,1)(2,2)--(3,1)(3,3)--(5,1)(4,4)--(7,1);
\draw(0.5,0.5)node[anchor=north west]{$1$}(2.5,0.5)node[anchor=north west]{$2$}
(4.5,0.5)node[anchor=north west]{$5$}(6.5,0.5)node[anchor=north west]{$3$}
(8.5,0.5)node[anchor=north west]{$4$};
\draw(2,1)node{$R_1$}(3.5,1.5)node{$R_2$}(5,2)node{$R_3$}(6.5,2.5)node{$R_4$};
}
\end{align*}
Then, we have 
\begin{align*}
Z(D)=\prod_{i=1}^{4}Y^{\downarrow}(\mu_i)=[4][3]^2[2],
\end{align*}
where $\mu_1=(1)$, $\mu_2=(2)$, $\mu_{3}=(2)$, and $\mu_4=(3)$.
\end{example}

Note that from Proposition \ref{prop:Ydet}, the generating function 
$Y^{\downarrow}(\mu_i)$ can be expressed as a determinant.
Therefore, $Z(L^{\downarrow},T)$ can be expressed as the product 
of $p$ determinants.
Recall that given a decreasing label $L^{\downarrow}_{0}$ and $T$,
we define a determinant $Y(L^{\downarrow}_{0},T)$ (see Eq. (\ref{eq:Ydet1})).
From Proposition \ref{prop:YDyck}, the determinant $Y(L^{\downarrow}_{0},T)$
is the generating function of Dyck tilings which contain only trivial 
Dyck tiles.

By comparing $Z(L^{\downarrow}_{0},T)$ with $Y(L^{\downarrow}_{0},T)$,
we obtain the following proposition.

\begin{prop}
The generating function $W(L^{\downarrow}_0,T)$ defined by 
\begin{align*}
W(L^{\downarrow}_{0},T):=Z(L^{\downarrow}_0,T)-Y(L^{\downarrow}_0,T),
\end{align*}
satisfies $W(L^{\downarrow}_0,T)\in\mathbb{N}[q]$.
The generating function $W(L^{\downarrow}_0,T)$ is the generating 
function of Dyck tilings which contain at least one $312$-patterns.
\end{prop}

For the Dyck tilings $D_1$ in Example \ref{ex:D1D2}, the generating 
function $Z(D_1)$ contains one non-trivial Dyck tiling. On the other hand,
$Z(D_2)$ contains two non-trivial Dyck tilings.
Similarly, the generating function $Z(D)$ in Example \ref{ex:D5} contains 
35 non-trivial Dyck tilings.

\section{Cover relation: revisited}
\label{sec:crrev}
In Section \ref{sec:LT}, we define a cover relation on decreasing labels 
of a tree $T$. 
The poset of decreasing labels is not in general a lattice. 
See Figure \ref{fig:LTHasse} as an example.
In this section, we consider another description of the cover relation 
in terms of a sequence of integers. 
We show that the poset of sequences is a lattice.

Let $L^{\downarrow}$ be a decreasing label of $T$.
We construct a sequence $\tau:=(\tau_1,\ldots,\tau_{n})$ of non-negative integers of length $n$
as follows.
The value $\tau_{i}$ is an integer given by
\begin{align}
\label{eq:deftau}
\tau_{i}:=2\cdot \#\{j>n+i-1 : e(n+1-i)\rightarrow e(j)\}
+\#\{j>n+1-i : e(n+1-i)\uparrow e(j)\}
\end{align}
where we define two symbols $\rightarrow$ and $\uparrow$ on edges by 
\begin{align*}
&``e(i)\rightarrow e(j)" \Leftrightarrow ``e(j) \text{ is strictly right to } e(i) \text{ in } T," \\
&``e(i)\uparrow e(j)" \Leftrightarrow ``e(j)\in p(e(i))," \\
\end{align*}
where $p(e(i))$ is a sequence of edges defined in Eq. (\ref{eq:defp}).
By construction, it is obvious that $0\le \tau_{i}\le 2(i-1)$ for all $1\le i\le n$.

We define the cover relation on an integer sequence $\tau$ as follows.
We write $\tau\lessdot \tau'$ if and only if 
there exists a pair $(i,j)$ of integers with $1\le i<j\le n$ such that 
\begin{enumerate}
\item $\tau_j\ge \tau_{i}+2$,
\item $\tau_{k}\ge \tau_{j}$ for $i+1\le k\le j-1$,
\item $\tau'_{i}=\tau_j-2$ and $\tau'_{j}=\tau_i$,
\item $\tau'_{k}=\tau_{k}$ if $k\neq i,j$.
\end{enumerate}
Since $\tau$ determines a label on $T$ uniquely, we have a bijection between 
a decreasing label $L^{\downarrow}$ and $\tau$.

\begin{prop}
\label{prop:tauL}
In the above notation, we have 
\begin{align*}
\tau\lessdot\tau'\Rightarrow L^{\downarrow}\lessdot L^{'\downarrow},
\end{align*}
where $L^{\downarrow}$ and $L^{'\downarrow}$ are the decreasing labels corresponding 
to $\tau$ and $\tau'$.
\end{prop}
\begin{proof}
Let $i<j$ be two integers.
We denote $\overline{i}:=n+1-i$.
Suppose $\tau\lessdot\tau'$ and they satisfy the conditions from (1) to (4).
We consider the two labels $\overline{i}$ and $\overline{j}$ in a decreasing label of $T$. 
If we have $\overline{i}$ is left to $\overline{j}$ a tree $T$, then we have 
$\tau_{i}\ge\tau_{j}$. This violates the condition (1).
Similarly, if we have $\overline{i}$ is contained in the sequence of edges $p(\overline{j})$,
then we have $\tau_{j}=\tau_{i}+1$, or there exists an integer $i<k<j$ such that 
$\tau_{k}\le\tau_{i}$.  This violates the condition (2).
Finally, we consider the case where $\overline{i}$ is right to $\overline{j}$.
Since $\overline{i}>\overline{j}$, we have $\tau_{j}\ge\tau_{i}+2$.
If there exists an integer $i<k<j$ such that the label $\overline{k}$ right to 
$\overline{j}$ and left to $\overline{i}$, then we have $\tau_{i}<\tau_{k}<\tau_{j}$.
This violates the condition (2).
From these, if $\tau$ and $\tau'$ satisfy the conditions (1) to (4),
then the label $\overline{j}$ is left to $\overline{i}$ in $T$ and 
there exists no $k$ such that $i<k<j$ and $\overline{k}$ is right to 
$\overline{j}$ and left to $\overline{i}$.
This means that the labels $\overline{i}$ and $\overline{j}$ can be 
exchanged in the decreasing label.
Therefore, we have $\tau\lessdot\tau'\Rightarrow L\lessdot L'$.
\end{proof}

\begin{remark}
In general, the converse of Proposition \ref{prop:tauL} is not true.
The cover relation on sequences $\tau$ of integers is weaker than 
the cover relation on decreasing labels. 
An example will be given in Example \ref{ex:posettault}.
\end{remark}

\begin{lemma}
\label{lemma:crtau}
Let $(i,j)$ be a pair of integers satisfying the conditions (1) and (2).
Then, $j$ is the label of a decreasing label $L^{\downarrow}$ satisfying the following conditions:
\begin{enumerate}[(a)]
\item $j>i$ and $j$ is right to $i$,
\item there exists no integer $j'$ such that $i<j'<j$ and $j'$ is right to $i$ and left to $j$, 
\item $j$ is the right-most among the labels satisfying (b).
\end{enumerate}
\end{lemma}
\begin{proof}
Since the pair $(i,j)$ satisfies the conditions (1) and (2), $(i,j)$ satisfies the conditions 
(a) and (b) by Proposition \ref{prop:tauL}.
Suppose the label $j'>j$ is right to $i$ and left to $j$ in $L^{\downarrow}$.  
By construction of $\tau$ in Eq. (\ref{eq:deftau}), 
we have $\tau_{n+1-j'}<\tau_{n+1-j}<\tau_{n+1-i}$. 
By the conditions (1) and (2), the pair $(i,j')$ is not admissible.
This implies that the pair $(i,j)$ satisfies the condition (c).
\end{proof}

The cover relation $\lessdot$ defines a poset $(\mathcal{S}(\tau),\lessdot)$ where 
$\mathcal{S}(\tau)$ is the set of sequences $\tau$.
Equivalently, we also have a poset $(\mathcal{S}(L^{\downarrow}),\lessdot)$ which 
is induced from the poset $(\mathcal{S}(\tau),\lessdot)$ by Proposition \ref{prop:tauL}.
Here, $\mathcal{S}(L^{\downarrow})$ is the set of decreasing labels.
We identify $(\mathcal{S}(\tau),\lessdot)$ with $(\mathcal{S}(L^{\downarrow}),\lessdot)$
by the natural bijection.

\begin{example}
\label{ex:posettault}
We consider the tree $T$ corresponding to $(UD)^{3}$.
In Figure \ref{fig:twoposets}, we depict two posets associated to the tree $T$.
The left poset is $(123,\lessdot)$ for decreasing labels defined in Definition \ref{defn:posetlt},
and the right poset is $(\mathcal{S}(024),\lessdot)$.
Note that we do not have an edge from $132$ to $312$ in the right poset.
\begin{figure}[ht]
\tikzpic{-0.5}{[yscale=0.5,xscale=0.8]
\node(0)at(0,0){$123$};
\node(1)at(2,2){$213$};
\node(2)at(2,-2){$132$};
\node(3)at(4,2){$312$};
\node(4)at(4,-2){$231$};
\node(5)at(6,0){$321$};
\draw(0)to(1)(0)to(2)(1)to(3)(1)to(4)(2)to(3)to(5)(2)to(4)to(5);
}\qquad
\tikzpic{-0.5}{[yscale=0.5,xscale=0.8]
\node(0)at(0,0){$024$};
\node(1)at(2,2){$022$};
\node(2)at(2,-2){$004$};
\node(3)at(4,2){$002$};
\node(4)at(4,-2){$020$};
\node(5)at(6,0){$000$};
\draw(0)to(1)(0)to(2)(1)to(3)(1)to(4)(3)to(5)(2)to(4)to(5);
}
\caption{Two posets $(123,\lessdot)$ and $(\mathcal{S}(024),\lessdot)$}
\label{fig:twoposets}
\end{figure}
\end{example}

\begin{prop}
\label{prop:lattice}
The poset $(\mathcal{S}(\tau),\lessdot)$, equivalently $(\mathcal{S}(L^{\downarrow}),\lessdot)$, 
is a graded bounded lattice.
\end{prop}

\begin{remark}
The poset of decreasing labels $(L^{\downarrow},\lessdot)$ is not a lattice 
in general. For example, consider the poset $(123,\lessdot)$ in Figure \ref{fig:twoposets}.
The elements $312$ and $231$ cover both of $213$ and $132$.
\end{remark}

Before proceeding to the proof of Proposition \ref{prop:lattice},
we introduce a lemma concerning the cover relation on a sequence $\tau$.

\begin{lemma}
\label{lemma:ij}
Fix an integer $j$. Then, there exists at most one $i<j$ such that 
the pair $(i,j)$ satisfies the conditions (1) to (4) for $\tau$.
\end{lemma}
\begin{proof}
Suppose that a pair $(i,j)$ satisfies the conditions (1) to (4).
Since $\tau_{k}\ge\tau_{j}$ for $i+1\le k\le j-1$, the pair $(k,j)$ does not 
satisfy the condition (1).
Consider the integer $i'<i$. Since $\tau_{i}<\tau_{j}$, the pair $(i',j)$
violates the condition (2).
From these, we have at most one pair $(i,j)$ for a given $j$. 
\end{proof}

Recall that we have a bijection between a sequence $\tau$ and  a decreasing label $L$.
By this bijection, we have an unrefineable sequence of labels
\begin{align}
\label{eq:pchainL}
L^{\downarrow}_{\min}=L_0\lessdot L_1\lessdot L_2\lessdot\ldots\lessdot L_{p}=L, 
\end{align}
where $L^{\downarrow}_{\min}$ is a decreasing label corresponding to the minimal integer 
sequence $\tau_{\min}$ of the poset.
By the cover relation on decreasing labels induced from the cover relation on $\tau$, 
it is obvious that the integer $p$ is uniquely determined given two labels $L$ 
and $L^{\downarrow}_{\mathrm{min}}$.
Note there may be several unrefineable sequences of length $p$. 
However, the number $p$ is not changed by a choice of a sequence. 
Then, we define the rank $\rho(L)$ of the label $L$ by $\rho(L):=p$.

\begin{lemma}
\label{lemma:tauone}
Let $\tau_1\neq\tau_2$ be two sequences satisfying $\rho(\tau_1)=\rho(\tau_2)\neq0$.
Then, there is no pair $(\tau,\tau')$ such that $\tau\neq\tau'$, and $\tau$ and $\tau'$ cover 
both $\tau_1$ and $\tau_2$.
\end{lemma}
\begin{proof}
We write $\tau\xrightarrow{(i,j)}\tau'$ if the pair $(i,j)$ satisfies the conditions (1) to (4).
Suppose that we have $\tau_1\lessdot\tau,\tau'$. 
This implies that we have $\tau_1\xrightarrow{(i,j)}\tau$ and $\tau_1\xrightarrow{(k,l)}\tau'$.
We write the decreasing label corresponding to $\tau$ by $L(\tau)$.
We consider the following  two cases: 1) $i, j, k$ and $l$ are all distinct, 
2) three of $i, j, k$ and $l$ are distinct.

Case 1). 
In a deceasing label corresponding to $\tau_1$, the label $j$ is right to $i$ and 
the label $l$ is right to $k$.
This implies that $j$ is left to $i$ and $l$ is right to $k$ in $L(\tau)$.
Similarly, since $\tau\neq\tau'$, the label $l$ is left to $k$ and $j$ is right to $i$ in $L(\tau')$.
Since $\tau_1\neq\tau_2$, we consider the two cases: a) the positions of $i$ and $j$ in $L(\tau)$
are the same as those of $L(\tau_2)$ respectively, and b) the position of either $i$ or $j$ in $L(\tau)$
is the same of that in $L(\tau_2)$.
In both cases, the positions of both $i$ and $j$ in $L(\tau_2)$ are different from those in $L(\tau')$.
The cover relation implies that there exists no pair $(i',j')$ such that 
$\tau_2\xrightarrow{(i',j')}\tau'$ since $j>i$ and $\rho(\tau')-\rho(\tau_2)=1$.

Case 2). We consider the positions of the labels $i,j$ and $k$ in $L(\tau_1)$.
We have six cases corresponding to the permutations in $\{i,j,k\}$ with $i<j<k$.
We write $x\rightarrow y$ if $y$ is right to $x$ in $L(\tau_1)$.
The first case is $k\rightarrow j\rightarrow i$. In this case, there is no element which
covers $\tau_1$.
The second case is $k\rightarrow i\rightarrow j$, $j\rightarrow k\rightarrow i$, and 
$i\rightarrow k\rightarrow j$. By Lemma \ref{lemma:crtau}, there exists only one 
element which covers $\tau_1$.
The third case is $j\rightarrow i\rightarrow k$.
Then, we have $j\rightarrow k\rightarrow i$ in $L(\tau)$ and $k\rightarrow i\rightarrow j$ in $L(\tau')$.
Since $\tau_1\neq\tau_2$ and $\tau$ covers $\tau_2$, a unique candidate 
for $\tau_2$ is $i\rightarrow k\rightarrow j$.
However, it is obvious that $\tau'$ does not cover $\tau_2$.
The last case is $i\rightarrow j\rightarrow k$ in $L(\tau_1)$. Then, we have 
$j\rightarrow i\rightarrow k$ in $L(\tau)$ and $i\rightarrow k\rightarrow j$ in $L(\tau')$.
As in the third case, there exists no $\tau_2$ such that $\tau_2\lessdot\tau,\tau'$.

In both cases, there is no two sequences $\tau$ and $\tau'$ which cover $\tau_1$ and $\tau_2$ at the 
same time. This completes the proof.
\end{proof}

\begin{proof}[Proof of Proposition \ref{prop:lattice}]
To show the poset is a lattice, we show that $\rho$ is indeed a rank function and show that there exists 
a unique join and meet given two sequences $\tau$ and $\tau'$.

From Proposition \ref{prop:tauL} and Lemma \ref{lemma:crtau}, if there is a sequence  (\ref{eq:pchainL}) 
of decreasing labels, then the number $p$ is uniquely fixed by $L_{\min}$ and $L$.
Therefore, $\rho$ is a rank function on the poset $(\mathcal{S}(\tau),\lessdot)$.

We will show that there exits a unique join and meet for given two sequences 
$\tau$ and $\tau'$.
Let $\tau:=(\tau_1,\ldots,\tau_n)$ and $\tau':=(\tau'_1,\ldots,\tau'_{n})$ be two 
sequences.
The join $\tau\vee\tau'$ is given by the following algorithm:
\begin{enumerate}
\item Set $j=n$.
\item If $\tau_{j}>\tau'_{j}$, take a pair $(i,j)$ such that $\tau_j\ge\tau_{i}+2$ and 
$\tau_{k}\ge\tau_{j}$ for $i+1\le k\le j-1$.
We have an element $\tau''$ such that $\tau\xrightarrow{(i,j)}\tau''$. 
Replace $\tau$ by $\tau''$ and repeat this procedure until we have $\tau_j=\tau'_{j}$.
If $\tau_{j}=\tau'_{j}$, then go to (3).
\item Decrease $j$ by one, and go to (2). The algorithm stops when $j=1$.
\end{enumerate}
This algorithm is well-defined since Lemma \ref{lemma:crtau} insures at least one 
pair $(i,j)$ if $\tau_{j}>\tau'_{j}$, and Lemma \ref{lemma:ij} insures at most one 
$(i,j)$ if it exists. Further, Lemma \ref{lemma:tauone} implies that if we have 
a join of two elements, then the join is unique.

We will show that there exists a unique meet of $\tau$ and $\tau'$.
Let $L(\tau)$ and $L(\tau')$ be decreasing labels corresponding to $\tau$ and $\tau'$.
We define $^{T}L(\tau)$ to be a mirror image of $L(\tau)$ along the vertical line.
Let $\nu$ and $\nu'$ be two sequences such that $^{T}L(\tau)=L(\nu)$ and $^{T}L(\tau')=L(\nu')$.
From Lemma \ref{lemma:crtau}, we have $\tau\lessdot\tau'\Leftrightarrow \nu'\lessdot\nu$.
Therefore, the join of $\nu$ and $\nu'$ is the meet of $\tau$ and $\tau'$.
Since the join of $\nu$ and $\nu'$ exists and it is unique, the meet of $\tau$ and $\tau'$ 
exists and it is unique. This completes the proof.
\end{proof}

\section{Rational Dyck tilings}
\label{sec:rDT}
\subsection{Generalized Dyck paths}
Let $\lambda$ be a Dyck path of size $n$. Recall that we have an expression 
of $\lambda$ in terms of a binary word consisting of $U$ and $D$.
We define two operations $\mathbb{U}^{k}$ and $\mathbb{D}^{k}$ on a binary word $\lambda$.
The operation $\mathbb{U}^{k}$ (resp. $\mathbb{D}^{k}$) replaces 
$U$ (resp. $D$) by $U^{k}$ (resp. $D^k$) in $\lambda$.
For example, we have $\mathbb{D}^{2}(UDU^2D^2)=UD^2U^2D^4$.
Note that we have $\mathbb{U}^{a}(\mathbb{D}^{b}(\lambda))=\mathbb{D}^{b}(\mathbb{U}^{a}(\lambda))$
for an up-right lattice path $\lambda$.

An $(a,b)$-Dyck path of size $n$ is a lattice path from $(0,0)$ to $(bn,an)$ such that it never 
goes below the line $y=bx/a$. 
We denote by $\mathcal{P}^{(a,b)}_{n}$ the set of $(a,b)$-Dyck paths of size $n$.
Especially, two classes of generalized Dyck paths, $(1,k)$-Dyck paths and $(k,1)$-Dyck paths,
are fundamental cases to consider $(a,b)$-Dyck paths in general. 
For example, we have twelve $(1,2)$-Dyck paths of size three:
\begin{align}
\label{eq:12Dyck3}
\begin{aligned}
&UDDUDDUDD \qquad UDDUDUDDD \qquad UDDUUDDDD \qquad UDUDDDUDD \\
&UDUDDUDDD \qquad UDUDUDDDD \qquad UDUUDDDDD \qquad UUDDDDUDD \\
&UUDDDUDDD \qquad UUDDUDDDD \qquad UUDUDDDDD \qquad UUUDDDDDD
\end{aligned}
\end{align}

We consider the subset $\mathbb{D}^{k}_{n}:=\{P: P=\mathbb{D}^{k}(\lambda), \lambda\in\mathcal{P}^{(1,1)}_{n} \}$ 
of $\mathcal{P}^{(1,k)}_{n}$.
Similarly, we define $\mathbb{U}_{n}^{k}:=\{P: P=\mathbb{U}^{k}(\lambda), \lambda\in\mathcal{P}^{(1,1)}_{n}\}$
as the set of $(k,1)$-Dyck paths.
The cardinality $|\mathbb{D}^{k}_{n}|=|\mathbb{U}^{k}_{n}|$ is equal to the Catalan number by its definition.
For example, we have five $(1,2)$-Dyck paths in Eq. (\ref{eq:12Dyck3}) which are in $\mathbb{D}^{2}_{3}$.
More generally, we define the set of $(a,b)$-Dyck paths  
$\mathbb{UD}^{(a,b)}_{n}:=\{P: P=\mathbb{U}^{a}(\mathbb{D}^{b}(\lambda)), \lambda\in\mathcal{P}^{(1,1)}_{n}\}$.

As in the case of Dyck paths, one can consider a Dyck tiling on these generalized 
Dyck paths.
In this paper, we only consider rational Dyck tilings above a generalized Dyck path
in $\mathbb{UD}^{(a,b)}_{n}$. Other Dyck tilings above an $(a,b)$-Dyck path which 
is not in $\mathbb{UD}^{(a,b)}_{n}$ do not behave nicely compared to  the case of 
$\mathbb{UD}^{(a,b)}_{n}$.
Recall that a Dyck tile of size $n\ge1$ is characterized by a Dyck path of size $n$.
Similarly, an $(a,b)$-Dyck tile of size $n$ is characterized by an $(a,b)$-Dyck path.
Since the bottom path of a Dyck tiling is in $\mathbb{UD}^{(a,b)}_{n}$, the Dyck path 
which characterizes a non-trivial Dyck tile is also in $\mathbb{UD}^{(a,b)}_{n}$.
In this way, one can have an $(a,b)$-Dyck tiling above an $(a,b)$-Dyck path.
Let $D^{(a,b)}$ be a Dyck tile of size $n\ge0$.
Then, we define the weight by
\begin{align}
\label{eq:defwtab}
\mathrm{wt}_{(a,b)}(D^{(a,b)}):=(a+b-1)n+1.
\end{align}
Let $\mathcal{D}$ be a set of $(a,b)$-Dyck tilings.
The generating function of $(a,b)$-Dyck tilings is defined to be
\begin{align*}
Z(\mathcal{D}):=
\sum_{D\in\mathcal{D}}q^{\mathrm{wt}_{(a,b)}(D)}.
\end{align*}

We first study the relation between $(1,k)$- and $(k,1)$-Dyck tilings and labeled trees.
Then, we study $(a,b)$-Dyck tilings through $(a,1)$- and $(1,b)$-Dyck tilings.

\subsection{\texorpdfstring{$(1,k)$}{(1,k)}-Dyck tilings}
$(1,k)$-Dyck tilings first appeared as a generalization of Dyck tilings in \cite{JVK16}.
In the following sections, we focus on $(1,b)$-Dyck tilings since they play a central 
role when we consider an $(a,b)$-Dyck tiling.

\subsubsection{Cover relation}
\label{sec:1kCL}
Let $\lambda$ be a Dyck path of size $n$ and $T$ a tree corresponding to $\lambda$.
We define $\lambda^{(k)}:=\mathbb{D}^{k}(\lambda)$, and $T^{(k)}$ as a tree which 
is obtained from $T$ by dividing an edge of $T$ into $k$ smaller edges.
By construction, $T^{(k)}$ has $kn$ edges.
Note that if $\lambda^{(k)}\in\mathcal{P}^{(1,k)}_{n}\setminus\mathbb{D}^{k}_{n}$, 
then we can not have a tree by uniformly dividing an edge of $T$ into $k$ smaller edges.

Since the number of up steps in $\lambda^{(k)}$ is $n$, each up step is bijective 
to an edge in $T$, and the set of $k$ edges in $T^{(k)}$.

Consider a Dyck tiling above $\lambda^{(k)}$ such that it contains only trivial 
Dyck tiles. 
Let $\mu:=(\mu_{1},\mu_2,\ldots,\mu_n)$ be an integer sequence where $\mu_i$ is 
$(n-i)k$ minus the number of Dyck tiles associated to the $i$-th up step from top.
We define the set $S_{i}$ of $k$ integers for $1\le i\le n$ from $\mu_{i}$ 
as follows.
\begin{enumerate}[(S1)]
\item Set $S=[1,nk]$ and $i=1$.
\item Let $j_1$ be the $\mu_{i}+1$-th smallest element in $S$. 
Then, $S_{i}$ is the set of $k$ successive increasing integers in $S$ starting 
from $j_1$. Namely, ``$j_i,j_{i+1}\in S_{i}$" $\Leftrightarrow$ ``an integer 
$j'\in[j_{i}+1,j_{i+1}-1]$ is in $\bigcup_{1\le l\le i}S_{l}$."
\item Replace $S$ by $S\setminus S_{i}$ and increase $i$ by one.
Then, go to (S2). The algorithm stops when $i=n$.
\end{enumerate}

By construction, one can assign the set $S_{i}$ to the $i$-th up step in $\lambda^{(k)}$.
Recall that each up step corresponds to an edge $e$ in $T$, and equivalently to the 
edges $e_{1},\ldots,e_{k}$ in $T^{(k)}$.
Therefore, we assign the integers in $S_{i}$ to the edges $e_1,\ldots,e_{k}$ such that 
$T^{(k)}$ is decreasing from the root to leaves.
This defines a decreasing label $L^{(k)\downarrow}$ on $T^{(k)}$.

Since $L^{(k)\downarrow}$ is decreasing, one can define a sequence of integers 
$\tau:=\tau(L^{(k)\downarrow})$ by Eq. (\ref{eq:deftau}).
We take a notation that the right-most step in $\lambda^{(k)}$ corresponds to the left-most edge 
in the planar rooted tree.
Below, we define a cover relation on the set of decreasing labels by use of 
$\tau$ instead of the label $L^{(k)\downarrow}$ itself.

Given a set $V=\{v_1,v_2,\ldots,v_{n}\}$ of integers, we define
\begin{align*}
\overline{V}:=\{nk+1-v_i |1\le i\le n \}.
\end{align*}

Let $\tau$ and $S_{i}, 1\le i\le n$, be a sequence and the sets as above.
We say that $\{\tau_{j_1},\tau_{j_2},\ldots,\tau_{j_k}\}$ is a block 
if $\{j_1,j_2,\ldots,j_{k}\}=\overline{S_{i}}$ for some $i\in[1,n]$.
If $\{\tau_{j_1},\tau_{j_2},\ldots,\tau_{j_k}\}$ is a block in $\tau$,
$\tau_q\ge\tau_{j_k}$ for $q\in[p+1,j_{k}-1]\setminus\{j_1,j_2,\ldots,j_{k}\}$, 
and $\tau_{p}\le \tau_{j}$,
we construct a new sequence $\tau'$ by 
\begin{enumerate}
\item $\tau'_{p}=\tau_{j_1}-2$,
\item $\tau'_{j_r}=\tau_{j_{r+1}}-2$ for $1\le r\le k-1$,
\item $\tau'_{j_{k}}=\tau_{p}$,
\item $\tau'_{s}=\tau_s$ for $s\in[1,nk]\setminus\{p,j_1,\ldots,j_{k}\}$.
\end{enumerate}
We say that $\{\tau'_{p},\tau'_{j_1},\tau'_{j_2},\ldots,\tau'_{j_{k-1}}\}$ 
is a block in $\tau'$.
Similarly, $\{\tau'_{j_{k}},\tau'_{p_1},\ldots,\tau'_{p_{k-1}}\}$ is a block 
in $\tau'$ if $\{\tau_{p},\tau_{p_1},\ldots,\tau_{p_{k-1}}\}$ is a block 
in $\tau$. 
When $\tau'$ is obtained from $\tau$ by the above process, 
we say that $\tau'$ covers $\tau$ and write it by $\tau\lessdot\tau'$.

As in the case of Dyck tilings, i.e., $k=1$ case, one can easily show that 
this cover relation on $\tau$ defines a poset of $(1,k)$-Dyck tilings.
By generalizing the proof of Proposition \ref{prop:lattice} to the case 
of $(1,k)$-Dyck tilings, this poset is shown to be a graded bounded 
lattice.
A sequence $\tau$ for a $(1,k)$-Dyck tiling naturally defines a rooted 
plane tree with $nk$ edges and its label. Therefore, by translating 
$\tau$ into a label, one can describe the cover relation on $\tau$ 
in terms of the label.
The main difference is that we involve $k+1$ edges to define 
the cover relation on the label. 
Note that in Definition \ref{def:coverL}, we involve two 
edges $e_{l}$ and $e_r$ to define the cover relation on a label.

\begin{example}
\label{ex:1k}
We consider $(2,1)$-Dyck tilings of size $3$.
We have a following lattice whose minimal element is $\tau=(0,2,3,6,7,1)$:
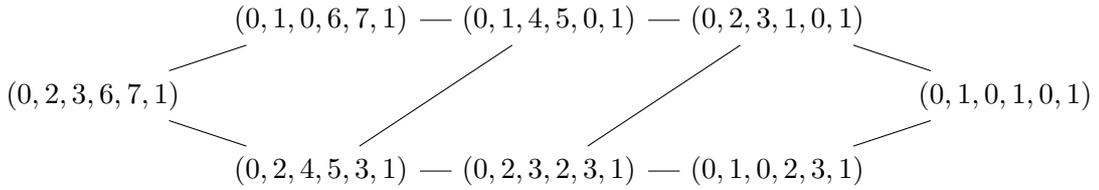
\begin{center}
\tikzpic{-0.5}{
\node (0) at (0,0) {$(0,2,3,6,7,1)$};
\node (1) at (3,1){$(0,1,0,6,7,1)$};
\node (2) at (3,-1){$(0,2,4,5,3,1)$};
\node (3) at (6,1){$(0,1,4,5,0,1)$};
\node (4) at (6,-1){$(0,2,3,2,3,1)$};
\node (5) at (9,1){$(0,2,3,1,0,1)$};
\node (6) at (9,-1){$(0,1,0,2,3,1)$};
\node (7) at (12,0){$(0,1,0,1,0,1)$};
\draw[-](0)--(1)(0)--(2)(1)--(3)(2)--(3)(2)--(4)(3)--(5)(4)--(5)(4)--(6)(5)--(7)(6)--(7);
}
\end{center}
For example, the sequences $(0,2,3,6,7,1), (0,1,0,6,7,1)$ and $(0,2,3,2,3,1)$
correspond to $(1,2)$-Dyck tilings 
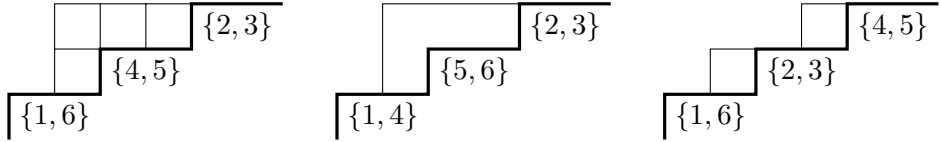
\begin{align*}
\tikzpic{-0.5}{[scale=0.6]
\draw[very thick](0,0)--(0,1)--(2,1)--(2,2)--(4,2)--(4,3)--(6,3);
\draw(1,1)--(1,3)--(4,3)(1,2)--(2,2)--(2,3)(3,3)--(3,2);
\draw(4,2.5)node[anchor=west]{$\{2,3\}$};
\draw(2,1.5)node[anchor=west]{$\{4,5\}$};
\draw(0,0.5)node[anchor=west]{$\{1,6\}$};
} \quad
\tikzpic{-0.5}{[scale=0.6]
\draw[very thick](0,0)--(0,1)--(2,1)--(2,2)--(4,2)--(4,3)--(6,3);
\draw(1,1)--(1,3)--(4,3);
\draw(4,2.5)node[anchor=west]{$\{2,3\}$};
\draw(2,1.5)node[anchor=west]{$\{5,6\}$};
\draw(0,0.5)node[anchor=west]{$\{1,4\}$};
} \quad
\tikzpic{-0.5}{[scale=0.6]
\draw[very thick](0,0)--(0,1)--(2,1)--(2,2)--(4,2)--(4,3)--(6,3);
\draw(1,1)--(1,2)--(2,2)(3,2)--(3,3)--(4,3);
\draw(4,2.5)node[anchor=west]{$\{4,5\}$};
\draw(2,1.5)node[anchor=west]{$\{2,3\}$};
\draw(0,0.5)node[anchor=west]{$\{1,6\}$};
} 
\end{align*}
A set of integers attached to the $i$-th up step from top is $S_{i}, 1\le i\le 3$. 

In this lattice, we have seven $(1,k)$-Dyck tilings which consist of only trivial 
Dyck tiles, and there is one Dyck tiling consisting of a non-trivial Dyck tile.
This unique Dyck tiling is corresponding to $\tau=(0,1,0,6,7,1)$.
\end{example}

\subsubsection{$k$-Stirling permutations}
In the case of $k=1$, we identify a sequence $\tau$ with a permutation 
in $[n]$.
For general $k$, we identify $\tau$ with a $k$-Stirling permutation 
instead of a permutation.
The correspondence between a $(1,k)$-Dyck tiling and a $k$-Stirling 
permutation appeared in \cite{JVK16}.

A {$k$-Stirling} permutation $s$ satisfies 
\begin{enumerate}[(s1)]
\item $s$ is a permutation of $\{1^k,2^k,\ldots,n^k\}$. 
\item Suppose $s_p=s_q$ for some $p<q$. Then, we have $s_{r}\ge s_{p}$
for $r\in[p+1,q-1]$.
\end{enumerate}
Suppose that $S_{i}=\{j_1,j_2,\ldots,j_{k}\}$.
Then, we define 
\begin{align}
\label{eq:defs}
s_{j_{m}}:=n+1-i, \quad 1\le m\le k.
\end{align}
The inverse map from $s$ to $S_{i}, 1\le i\le n$ is given by 
\begin{align*}
S_{i}=\{1\le p\le nk : s_{p}=n+1-i\}.
\end{align*}

\begin{remark}
When $S_{i}=\{(i-1)k+1,(i-1)k+2\ldots, ik\}$ for all $1\le i\le n$,
we have 
\begin{align*}
s=(\underbrace{n,\ldots,n}_{k},\underbrace{n-1,\ldots,n-1}_{k},\ldots,\underbrace{1,\ldots,1}_{k}).
\end{align*}
This choice of notation is compatible with a decreasing label.
\end{remark}

\begin{prop}
The definition (\ref{eq:defs}) is well-defined, i.e., $s$ is 
a $k$-Stirling permutation.
\end{prop}
\begin{proof}
Let $\{S_{i} : 1\le i\le n\}$  be the collection of the sets defined in Section \ref{sec:1kCL}.
The construction of $s$ from $S_{i}, 1\le i\le n$, implies that $s$ is a permutation of 
the set $\{1^k,2^k,\ldots,n^k\}$. Thus, it satisfies the condition (s1).
Since we construct $s$ by Eq. (\ref{eq:defs}) and we have the condition (S2) in the 
definition of $S_{i}$, it is obvious that $s$ satisfies the condition (s2).
This completes the proof.
\end{proof}

Let $s$ be a $k$-Stirling permutations satisfying (s1) and (s2).
Let $j(i)$ be the smallest integer such that $s_{j(i)}=i$ for $1\le i\le n$.
We define $\mu'_{i}, 1\le i\le n$, by
\begin{align}
\label{eq:defmudash}
\mu'_{i}:=\{1\le p\le j(i) : s_{p}<s_{j(i)}\}-1.
\end{align}
Then, we obtain a sequence of non-negative integers $\mu':=(\mu'_1,\mu'_2,\ldots,\mu'_{n})$.

\begin{prop}
Let $\mu$ be a sequence of integers as in Section \ref{sec:1kCL} and $\mu'$ as above.
Then, we have $\mu=\mu'$.
\end{prop}
\begin{proof}
Let $S_{i}=\{j_1<j_2<\ldots<j_{k}\}$.
The relation (\ref{eq:defs}) between a $k$-Stirling permutation and a collection of 
the sets $S_{i}$ implies that we have $s_{j_1}=n+1-i$ and $s_{j_1}$ is the left-most 
$n+1-i$ in $s$. 
Since the number $\mu'_{i}$ counts the number of $p$ satisfying (\ref{eq:defmudash}), 
it is obvious that we have  $\mu_{i}=\mu'_{i}$.
From these, we have $\mu=\mu'$.	
\end{proof}

\subsubsection{Decompositions}
We decompose a $(1,k)$-Dyck tiling into $k$ Dyck tilings.
The decomposition of a $(1,k)$-Dyck path above $\lambda^{(k)}$ into $k$ Dyck paths 
is studied in \cite{Shi21b}. We generalize this decomposition to the case of $(1,k)$-Dyck 
tilings which contain non-trivial Dyck tiles following \cite{Shi21c,Shi21b}.

Let $\mu=(\mu_1,\ldots, \mu_{n})$ be an integer sequence as in Section \ref{sec:1kCL}.
We construct $k$ integer sequences $\alpha_{1},\ldots,\alpha_{k}$ from 
$\mu$ as follows \cite{Shi21c,Shi21b}:
\begin{align}
\label{eq:dec}
\alpha_{i}:=\left\lfloor\genfrac{}{}{}{}{\mu-\sum_{j=1}^{i-1}\alpha_{j}}{k+1-i}\right\rfloor,
\end{align}
where $\lfloor x\rfloor$ is the floor function.
Recall that $\mu$ is defined above a $(1,k)$-Dyck path $\mathbb{D}^{k}(\lambda)$.
Since we consider only $(1,k)$-Dyck paths of the form $\mathbb{D}^{k}(\lambda)$,
we have $k$ Dyck tilings above the same Dyck path $\lambda$, and the top paths are characterized 
by $\alpha_{i}$, $1\le i\le k$.

We write $\alpha\subseteq_{H}\beta$ if two sequences $\alpha:=(\alpha_1,\ldots,\alpha_n)$ and 
$\beta:=(\beta_1,\ldots,\beta_n)$ satisfy 
\begin{align}
\label{eq:alphabeta}
\beta_i-1\le \alpha_i\le \beta_i.
\end{align}
Then, by construction of $\alpha_{i}, 1\le i\le k$, it is obvious that 
the sequences $\alpha_i$ satisfy  
\begin{align}
\label{eq:alphaad}
\alpha_i\subseteq_{H}\alpha_j, \quad\text{ for } i<j.
\end{align}
Given $k$ integers sequence $\{\alpha_i: 1\le i\le k\}$, we say that 
the sequences are {\it admissible} if they satisfy the condition 
(\ref{eq:alphaad}).

Since each $\alpha_i$ defines a Dyck tiling $D_{i}$ above the path $\lambda$,
the condition (\ref{eq:alphaad}) immediately implies that the decreasing labels $L_{i}^{\downarrow}$ 
for the Dyck tilings $D_{i}$ satisfy
\begin{align}
\label{eq:Ldad}
L_1^{\downarrow}\le L_2^{\downarrow}\le\ldots\le L_k^{\downarrow}. 
\end{align}
Actually, the condition (\ref{eq:alphaad}) is stronger than the 
condition (\ref{eq:Ldad}).

\begin{example}
The following $(1,3)$-Dyck tiling can be decomposed into 
three Dyck tilings.
\begin{center}
\tikzpic{-0.5}{[scale=0.5]
\draw[very thick](0,0)--(0,1)--(3,1)--(3,2)--(6,2)--(6,3)--(9,3);
\draw(1,1)--(1,2)--(3,2)--(3,3)(2,1)--(2,3)--(6,3)(4,2)--(4,3)(5,2)--(5,3);
}
\quad $\Longleftrightarrow$\quad 
\tikzpic{-0.5}{[scale=0.5]
\draw[very thick](0,0)--(0,1)--(1,1)--(1,2)--(2,2)--(2,3)--(3,3);
\draw(1,2)--(1,3)--(2,3);
\draw(0.5,0.5)node{$3$}(1.5,1.5)node{$1$}(2.5,2.5)node{$2$};
}
$\subseteq$\tikzpic{-0.5}{[scale=0.5]
\draw[very thick](0,0)--(0,1)--(1,1)--(1,2)--(2,2)--(2,3)--(3,3);
\draw(0,1)--(0,2)--(1,2)--(1,3)--(2,3);
\draw(0.5,0.5)node{$1$}(1.5,1.5)node{$3$}(2.5,2.5)node{$2$};
}
$\subseteq$
\tikzpic{-0.5}{[scale=0.5]
\draw[very thick](0,0)--(0,1)--(1,1)--(1,2)--(2,2)--(2,3)--(3,3);
\draw(0,1)--(0,3)--(2,3)(0,2)--(1,2)--(1,3);
\draw(0.5,0.5)node{$1$}(1.5,1.5)node{$2$}(2.5,2.5)node{$3$};
}
\end{center}
The decreasing labels corresponding to three Dyck tilings satisfy 
$213\le 231\le 321$. 
Here, we read the labels in a tree from right to left.
\end{example}

\subsection{\texorpdfstring{$(k,1)$}{(k,1)}-Dyck tilings}
As in the case of $(1,k)$-Dyck tilings, we set $\nu^{(k)}:=\mathbb{U}^{k}(\lambda)$ 
and $T^{(k)}$ as a tree with $nk$ edges.
Instead of a decreasing label, we consider an increasing label $L^{(k)\uparrow}$ on
$T^{(k)}$.
Suppose that a Dyck tiling above $\nu^{(k)}$ contains only trivial Dyck tiles.
Let $\xi_{i}$ be the number of Dyck tiles associated to the $i$-th down step 
from left to right. 
Here, we count the number of Dyck tiles by a vertical Hermite history.
We define the set $Q_{i}$ of $k$ integers for $1\le i\le n$ from $\xi_{i}$
as follows:
\begin{enumerate}
\item Set $Q:=[nk]$ and $i=1$.
\item Let $x_{i}$ be the $\xi_{i}+1$-th largest element in $Q$.
Then, $Q_{i}$ is the set of $k$ successive decreasing integers from $x_{i}$ 
in $Q$. Namely, ``$x_j,x_{j+1}\in Q_{i}$" $\Rightarrow$ ``$y\in[x_{j+1}+1,x_{j}-1]$ is 
in $\bigcup_{1\le l\le i}Q_{l}$."
\item  Replace $Q$ by $Q\setminus Q_{i}$, and increase $i$ by one. Then go to (2).
The algorithm stops when $i=n$.
\end{enumerate}
As in the case of $(1,k)$-Dyck tilings, we have an increasing label $L^{(k)\uparrow}$ on $T^{(k)}$.
We take a notation that the left-most edge in $\nu^{(k)}$ corresponds to 
the left-most edge in $T^{(k)}$.

Given an increasing label $L^{(k)\uparrow}$, we construct an integer sequence $\eta:=\eta(L^{(k)\uparrow})$ 
as follows.
\begin{align*}
\eta_i:=2\cdot \#\{j<i : e(j)\rightarrow e(i) \}+\#\{j<i: e(i)\uparrow e(j)\},
\end{align*}
where $e(i)$ is the edge in $L^{(k)\uparrow}$ whose label is $i$.

Then, we define a covering relation on the set of sequences $\eta$.
Let $Q_{i}, 1\le i\le n$, be the sets as above.
We say that $\eta_{j_1},\ldots,\eta_{j_{k}}$ is a block if $\{j_1,\ldots,j_{k}\}=Q_{i}$ for some 
$i\in[n]$.
Suppose that $\eta_{j_1},\eta_{j_2},\ldots,\eta_{j_{k}},\eta_{p}$ satisfy
\begin{enumerate}
\item $j_1<j_2<\ldots<j_{k}<p$ and $\eta_{j_1},\ldots,\eta_{j_{k}}$ is a block in $\eta$.
\item $\eta_{q}\ge\max\{\eta_{j_1},\ldots,\eta_{j_{k}},\eta_{p}\}$
for all $q\in[j_1,p]\setminus\{j_1,\ldots,j_{k},p\}$.
\item $\eta_{p}\ge \max\{\eta_{j_1},\ldots,\eta_{j_{k}}\}$ and $\eta_{p}\ge 2k$.
\end{enumerate}
Then, we define a new sequence $\eta'$ of $nk$ integers by 
\begin{enumerate}
\item $\eta'_{j_1}=\eta_{p}-2k$,  
\item $\eta'_{j_r}=\eta_{j_{r-1}}$ for $r\in[2,k]$ and $\eta'_{p}=\eta_{j_{k}}$,
\item $\eta'_{q}=\eta_q$ for all $q\in[nk]\setminus\{j_1,\ldots,j_k,p\}$.
\end{enumerate}
When $\eta'$ is obtained from $\eta$ by the above process, we say that 
$\eta'$ covers $\eta$ and write it by $\eta\lessdot\eta'$.

This cover relation $\lessdot$ naturally defines a poset on the 
sequences $\eta$. As in the case of $(1,k)$-Dyck tiling, 
one can show that this poset is a graded bounded lattice. 
In addition, one can translate the cover relation in terms 
of a label on a $(k,1)$-Dyck path $\mathbb{U}^{k}(\lambda)$.

\begin{example}
\label{ex:k1}
Set $k=2$. We consider the trivial $(2,1)$-Dyck tiling above $(U^2D)^3$ and below $U^5D^2UD$. 
This Dyck tiling gives the integer sequence $(0,0,1,4,5,9)$.
We have the following lattice on $(2,1)$-Dyck tilings.
\begin{center}
\tikzpic{-0.5}{
\node(0)at(0,0){$(0,0,1,4,5,9)$};
\node (1) at (3,1){$(0,0,1,5,4,5)$};
\node (2) at (3,-1){$(0,0,0,1,5,9)$};
\node (3) at (6,1){$(0,1,0,1,4,5)$};
\node (4) at (6,-1){$(0,0,1,0,1,9)$};
\node (5) at (9,1){$(0,1,0,0,1,5)$};
\node (6) at (9,-1){$(0,0,1,5,0,1)$};
\node (7) at (12,0){$(0,1,0,1,0,1)$};
\draw[-](0)--(1)(0)--(2)(1)--(3)(2)--(3)(2)--(4)(3)--(5)(4)--(5)(4)--(6)(5)--(7)(6)--(7);
}
\end{center}
Let $\mathbf{Q}(\eta):=(Q_1,Q_2,Q_3)$ be the collection of the sets $Q_{i}, 1\le i\le 3$, 
for $\eta$. In this example, we have
\begin{align*}
\mathbf{Q}((0,0,1,4,5,9))&=\{\{3,2\},\{5,4\},\{6,1\}\}, \\
\mathbf{Q}((0,0,1,5,4,5))&=\{\{3,2\},\{6,5\},\{4,1\}\}, \\
\mathbf{Q}((0,0,1,0,1,9))&=\{\{5,4\},\{3,2\},\{6,1\}\}.	
\end{align*}
\end{example}

As in the case of $(1,k)$-Dyck tilings, we decompose a 
$(k,1)$-Dyck tiling into $k$ Dyck tilings.
We replace $\mu$ by $\xi$ in Eq. (\ref{eq:dec}).
Then, we have $k$ integer sequences $\alpha_{i}, 1\le i\le k$.
We write $\alpha\subseteq_{V}\beta$ if $\alpha:=(\alpha_1,\ldots,\alpha_n)$ and 
$\beta:=(\beta_1,\ldots,\beta_{n})$ satisfy Eq. (\ref{eq:alphabeta}).
Then, we have 
\begin{align*}
\alpha_{i}\subseteq_{V}\alpha_{j},	
\end{align*}
for $i<j$.

\subsection{Duality between \texorpdfstring{$(1,k)$}{(1,k)}- and \texorpdfstring{$(k,1)$}{(k,1)}-Dyck tilings}
Let $D$ be a $(1,k)$-Dyck tiling, and $L^{(k)\downarrow}$ a decreasing 
label corresponding to $D$ on the tree $T^{(k)}$.
We denote by $^{T}D$ the transposed 
$(k,1)$-Dyck tiling. Namely, $^{T}D$ is the mirror image of $D$ 
along the line $y=-x$.
Similarly, we define a transposed decreasing label as $^{T}L^{(k)\downarrow}$.
Let $L^{(k)\uparrow}$ be an increasing label corresponding 
to $^{T}D$ on $T^{(k)}$.

We define an operation $\phi$ which maps a decreasing label to an increasing label.
Let $T$ be a planar rooted tree. 
When the root of $T$ has $m$ edges below it, we decompose $T$ into $m$ small trees 
$T_{1},\ldots,T_{m}$ where $T_{i}$ is a tree whose root has a single edge below it.
We denote it by a concatenation $T=T_{1}\circ T_{2}\circ\cdots\circ T_{m}$.
In other words, $T$ is a concatenation of prime trees $T_i$, $1\le i\le m$.
Suppose that $L_{i}^{\downarrow}$ (resp. $L_{i}^{\uparrow}$) is a decreasing (resp. increasing) label on $T_{i}$.
The map $\phi:L_{i}^{\downarrow}\mapsto L_{i}^{\uparrow}$ is defined as follows.
\begin{enumerate}
\item We first standardize the labels in $L_{i}$. If $L_{i}$ has $n$ edges, 
we replace the $i$-th smallest label $l_{i}$ in $L_{i}$ by $i$. 
\item We replace $i$ by $n+1-i$ and obtain a standardized increasing tree.
\item We destandardize the increasing label, that is, we replace $i$ by $l_{i}$.
\end{enumerate}
Then, we define 
\begin{align}
\label{ex:defphi}	
\phi(T):=\phi(T_{1})\circ\phi(T_{2})\circ\cdots\circ\phi(T_{m}).
\end{align}

Then, $L^{(k)\uparrow}$ and $L^{(k)\downarrow}$ is related as follows.
\begin{prop}
\label{prop:dual}
In the above notation, we have
\begin{align}
\label{eq:Ldual}
L^{(k)\uparrow}=\phi({\ }^{T}L^{(k)\downarrow}),
\end{align}
where $\phi$ is defined by Eq. (\ref{ex:defphi}).
\end{prop}
\begin{proof}
Since $^{T}D$ is the mirror image of $D$, the increasing label $L^{(k)\uparrow}$
is a label on the transposed tree of $T^{(k)}$.
Suppose $T^{(k)}$ can be decomposed into $m$ subtrees, and denote it 
by $T^{(k)}=T_1\circ\cdots\circ T_{m}$.
When $T_i$ consists of $kr$ edges, we assign a label on the edge in $T_{i}$
from a collection of $r$ sets of integers $\{S_{1},\ldots, S_{r}\}$.
By construction of $S_{i}$, we have a decreasing label on $T_{i}$.
It is obvious that the action of $\phi$ on $T_{i}$ gives an increasing 
label on $T_{i}$.
The set $Q_{i}$ is constructed from the number of Dyck tiles in a vertical 
Hermite history. On the other hand, $S_{i}$ is constructed from $(n-i)k$ 
minus the number of Dyck tiles in a horizontal Hermite history.
Note that since we consider the transposed Dyck tilings, a horizontal Hermite history for 
$S_{i}$ corresponds to a vertical Hermite history for $Q_{i}$.
This implies that $S_{i}$ and $Q_{i}$ is a complement to each other.
This complementarity can be realized in the second step (2) in the definition of 
$\phi$.  
From these, if we construct an increasing label from $Q_{i}$'s,
we obtain $\phi(T_{i})$.
An increasing label can be constructed from $m$ increasing labels by 
a concatenation, we have Eq. (\ref{eq:Ldual}).
\end{proof}

\begin{example}
The Dyck tilings in Examples \ref{ex:1k} and \ref{ex:k1} are dual to 
each other. Thus, the two lattices are isomorphic.

In this case, the tree $T^{(2)}$ can be decomposed into $3$ subtrees which 
have two edges and a unique leaf. Then, the action of $\phi$ on the labeled trees gives 
the following consequence: $S_{i}=\phi(Q_{i})$. Here, $\phi$ simply replaces 
an increasing sequence by a decreasing sequence.  
\end{example}

\begin{example}
We consider the $(1,2)$- and $(2,1)$-Dyck tilings of size four as in Figure \ref{fig:1kk1}.
\begin{figure}[ht]
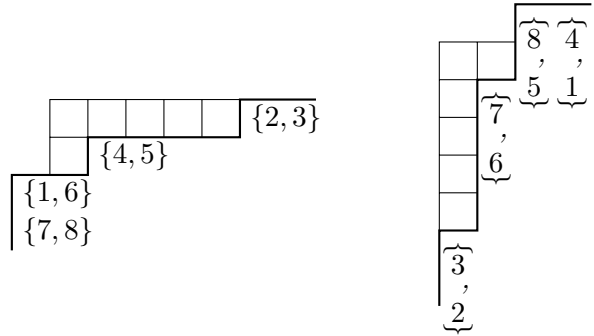

\tikzpic{-0.5}{[scale=0.5]
\draw[thick](0,0)--(0,2)--(2,2)--(2,3)--(6,3)--(6,4)--(8,4);
\draw(1,2)--(1,4)--(6,4)(1,3)--(2,3)--(2,4)(3,3)--(3,4)(4,3)--(4,4)(5,3)--(5,4);
\draw(6,3.5)node[anchor=west]{$\{2,3\}$}(2,2.5)node[anchor=west]{$\{4,5\}$};
\draw(0,1.5)node[anchor=west]{$\{1,6\}$}(0,0.5)node[anchor=west]{$\{7,8\}$};
}
\qquad
\tikzpic{-0.5}{[scale=0.5]
\draw[thick](0,0)--(0,2)--(1,2)--(1,6)--(2,6)--(2,8)--(4,8);
\draw(0,2)--(0,7)--(2,7)(0,3)--(1,3)(0,4)--(1,4)(0,5)--(1,5)(0,6)--(1,6)--(1,7);
\draw(0.5,2)node[anchor=north]{
\rotatebox{-90}{$\{\rotatebox{90}{$3$}\ \rotatebox{90}{\ ,}\ \rotatebox{90}{$2$}\}$}};
\draw(1.5,6)node[anchor=north]{
\rotatebox{-90}{$\{\rotatebox{90}{$7$}\ \rotatebox{90}{\ ,}\ \rotatebox{90}{$6$}\}$}};
\draw(2.5,8)node[anchor=north]{
\rotatebox{-90}{$\{\rotatebox{90}{$8$}\ \rotatebox{90}{\ ,}\ \rotatebox{90}{$5$}\}$}};
\draw(3.5,8)node[anchor=north]{
\rotatebox{-90}{$\{\rotatebox{90}{$4$}\ \rotatebox{90}{\ ,}\ \rotatebox{90}{$1$}\}$}};
}
\caption{Two $(1,2)$- and $(2,1)$-Dyck tilings of size $4$.}
\label{fig:1kk1}	
\end{figure}
In Figure \ref{fig:1kk1tree}, we depict the decreasing and increasing labeled trees 
for the Dyck tilings in Figure \ref{fig:1kk1}.
\begin{figure}[ht]
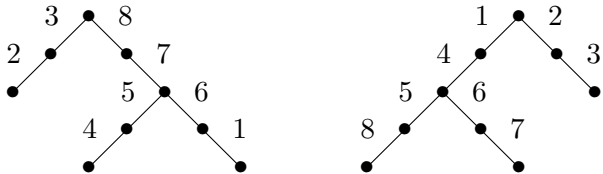

\tikzpic{-0.5}{[scale=0.5]
\draw(0,0)--(-2,-2)(0,0)--(4,-4)(2,-2)--(0,-4);
\draw(0,0)node{$\bullet$}(-1,-1)node{$\bullet$}(-2,-2)node{$\bullet$}(1,-1)node{$\bullet$}
(2,-2)node{$\bullet$}(3,-3)node{$\bullet$}(4,-4)node{$\bullet$}(1,-3)node{$\bullet$}(0,-4)node{$\bullet$};
\draw(-0.5,-0.5)node[anchor=south east]{$3$}(-1.5,-1.5)node[anchor=south east]{$2$}
(0.5,-0.5)node[anchor=south west]{$8$}(1.5,-1.5)node[anchor=south west]{$7$}
(2.5,-2.5)node[anchor=south west]{$6$}(3.5,-3.5)node[anchor=south west]{$1$}
(1.5,-2.5)node[anchor=south east]{$5$}(0.5,-3.5)node[anchor=south east]{$4$};
}\qquad
\tikzpic{-0.5}{[scale=0.5]
\draw(0,0)--(-4,-4)(0,0)--(2,-2)(-2,-2)--(0,-4);	
\draw(0,0)node{$\bullet$}(-1,-1)node{$\bullet$}(-2,-2)node{$\bullet$}
(-3,-3)node{$\bullet$}(-4,-4)node{$\bullet$}(1,-1)node{$\bullet$}(2,-2)node{$\bullet$}
(-1,-3)node{$\bullet$}(0,-4)node{$\bullet$};
\draw(-0.5,-0.5)node[anchor=south east]{$1$}(-1.5,-1.5)node[anchor=south east]{$4$}
(-2.5,-2.5)node[anchor=south east]{$5$}(-3.5,-3.5)node[anchor=south east]{$8$}
(-1.5,-2.5)node[anchor=south west]{$6$}(-0.5,-3.5)node[anchor=south west]{$7$}
(0.5,-0.5)node[anchor=south west]{$2$}(1.5,-1.5)node[anchor=south west]{$3$};
}
\caption{Two labeled trees: the left is decreasing an the right is increasing.}
\label{fig:1kk1tree}
\end{figure}
\end{example}

\subsection{\texorpdfstring{$(a,b)$}{(a,b)}-Dyck tilings}
\label{sec:divab}
Let $D$ be a trivial $(a,b)$-Dyck tiling above an $(a,b)$-Dyck path 
$\mathbb{U}^{a}(\mathbb{D}^{b}(\lambda))$.
We decompose the trivial $(a,b)$-Dyck tiling $D$ into $ab$ Dyck tilings
above $\lambda$.
Let $D_{i,j}$ with $1\le i\le a$ and $1\le j\le b$ be 
the $ab$ Dyck tilings, and let $\alpha_{i,j}$ be the integer 
sequence associated to the Dyck tiling $D_{i,j}$.
Then, the $ab$ sequences $\alpha_{i,j}$ satisfy 
\begin{align}
\label{eq:aii1}
\begin{split}
&\alpha_{i,j}\subseteq_{H} \alpha_{i,j'}, \\
&\alpha_{i,j}\subseteq_{V} \alpha_{i',j},
\end{split}
\end{align}
where $1\le j<j'\le b$ and $1\le i<i'\le a$.
We place $\alpha_{i,j}$ in the plane in the matrix notation.
Then, the condition (\ref{eq:aii1}) implies that 
the sequences in each row and each column are admissible. 
Thus, we have $ab$ admissible conditions on $\alpha_{i,j}$.
If there is a non-trivial Dyck $d_1$ tile of size $n$, then we have 
a non-trivial Dyck tile $d_1$ of size $n$ in $D_{a,b}$.
We fix the position of $d_2$ such that $d_2$ is compatible with the position 
of $d_1$.

Let $\mathrm{wt}_{(a,b)}(D)$ be the weight of $(a,b)$-Dyck tile $D$ 
of size $m$ defined by Eq. (\ref{eq:defwtab}). 
\begin{prop}
We have
\begin{align*}
\mathrm{wt}_{(a,b)}(D)=\sum_{i=1}^{a}\sum_{j=1}^{b}\mathrm{wt}_{(1,1)}(D_{i,j}).
\end{align*}
\end{prop}
\begin{proof}
We decompose an $(a,b)$-Dyck tile $D$ of size $m$ into $ab$ Dyck tilings $D_{i,j}$ 
for $i\in[1,a]$ and $j\in[1,b]$.  
By definition of a Dyck tile, $D$ is a ribbon. This means that
\begin{enumerate}
\item $D_{i,j}=\emptyset$ if $i\neq a$ and $j\neq b$; 
\item $D_{i,j}$ is a Dyck tile of size $m-1$ for $i=a$ or $j=b$ except $(i,j)=(a,b)$;
\item $D_{a,b}$ is a Dyck tile of size $m$.
\end{enumerate}
From these, we calculate 
\begin{align*}
\sum_{i,j}\mathrm{wt}_{(1,1)}(D_{i,j})&=m(a+b-2)+m+1, \\
&=\mathrm{wt}_{(a,b)}(D),
\end{align*}
which completes the proof.
\end{proof}

\begin{example}
\label{ex:Dyck23}
Consider a $(2,3)$-Dyck tiling of size $3$ consisting of 
single non-trivial Dyck tile of size one.
The decomposition of this Dyck tiling is given by 
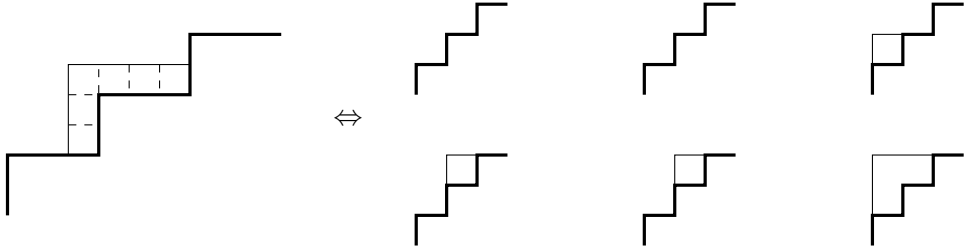
\begin{center}
\tikzpic{-0.5}{[scale=0.4]
\draw[very thick](0,0)--(0,2)--(3,2)--(3,4)--(6,4)--(6,6)--(9,6);
\draw(2,2)--(2,5)--(6,5);
\draw[dashed](2,3)--(3,3)(2,4)--(3,4)--(3,5)(4,5)--(4,4)(5,5)--(5,4);
}
\quad $\Leftrightarrow$
\tikzpic{-0.5}{
\node (0)at(0,0){
\tikzpic{-0.5}{[scale=0.4]
\draw[very thick](0,0)--(0,1)--(1,1)--(1,2)--(2,2)--(2,3)--(3,3);
}
};
\node (1)at(3,0){
\tikzpic{-0.5}{[scale=0.4]
\draw[very thick](0,0)--(0,1)--(1,1)--(1,2)--(2,2)--(2,3)--(3,3);
}
};
\node (2)at(6,0){
\tikzpic{-0.5}{[scale=0.4]
\draw[very thick](0,0)--(0,1)--(1,1)--(1,2)--(2,2)--(2,3)--(3,3);
\draw(0,1)--(0,2)--(1,2);
}
};
\node (3)at(0,-2){
\tikzpic{-0.5}{[scale=0.4]
\draw[very thick](0,0)--(0,1)--(1,1)--(1,2)--(2,2)--(2,3)--(3,3);
\draw(1,2)--(1,3)--(2,3);
}
};
\node (4)at(3,-2){
\tikzpic{-0.5}{[scale=0.4]
\draw[very thick](0,0)--(0,1)--(1,1)--(1,2)--(2,2)--(2,3)--(3,3);
\draw(1,2)--(1,3)--(2,3);
}
};
\node (5)at(6,-2){
\tikzpic{-0.5}{[scale=0.4]
\draw[very thick](0,0)--(0,1)--(1,1)--(1,2)--(2,2)--(2,3)--(3,3);
\draw(0,1)--(0,3)--(2,3);
}
};
}
\end{center}
Each row and column consists of admissible Dyck tilings.
The weight of this $(2,3)$-Dyck tiling is five since its size is one.
From the decomposition, the sum of the weights of $(1,1)$-Dyck tilings
is given by $0+0+1+1+1+2=5$. 
\end{example}

\subsection{Decomposition of \texorpdfstring{$(a,b)$}{(a,b)}-Dyck tilings}
Fix an integer $a\ge2$.
In this section, we consider an enumeration of $(a,b)$-Dyck tilings above 
the path $\mathbb{UD}_{n}^{(a,b)}$.
In Section \ref{sec:divab}, we decompose an $(a,b)$-Dyck tiling into 
$ab$ $(1,1)$-Dyck tilings. This decomposition reflects the weight 
of an $(a,b)$-Dyck tiling, however, we have several constraints (\ref{eq:aii1})
on $ab$ Dyck tilings.
To reduce the number of constraints, we decompose an $(a,b)$-Dyck tiling
into $a$ $(1,b)$-Dyck tilings.
This decomposition is well-defined since the lower path of an $(a,b)$-Dyck 
tiling is in $\mathbb{UD}_{n}^{(a,b)}$.

Let $V_{i}:=(v^i_{1},v^{i}_{2},\ldots,v^{i}_{bn})$, $1\le i\le a$, 
be the vertical Hermite histories corresponding to $a$ trivial 
$(1,b)$-Dyck tilings $D_i$, $1\le i\le a$.
By the decomposition of an $(a,b)$-Dyck tiling, we have 
the constraints 
\begin{align}
\label{eq:condvij}
v^{k}_{i}\le v^{j}_{i}\le v^{k}_{i}+1,  
\end{align}
where $i\in[1,bn]$ and $1\le j<k\le a$.
This constraint is equivalent to the condition (\ref{eq:alphaad}) for 
a horizontal Hermite history.

\begin{remark}
Two remarks are in order.
\begin{enumerate}
\item We have two decompositions of $(a,b)$-Dyck tilings: one is into $ab$ $(1,1)$-Dyck tilings, and 
the other is into $a$ $(1,b)$-Dyck tilings. The advantage of the first one is that the weight is preserved 
by the decomposition. The advantage of the second decomposition is that we involve the condition 
(\ref{eq:condvij}) on vertical Hermite histories. This is simpler than the first one, since we involve 
two types of conditions (\ref{eq:aii1}) on both horizontal and vertical Hermite histories in the first case.	
\item
The cover relation on vertical Hermite histories defines a poset structure on them.
By construction, this poset includes the poset $\mathcal{P}$ of $(a,b)$-Dyck paths, which correspond to
trivial $(a,b)$-Dyck tilings, as a subposet. 
Suppose two $(a,b)$-Dyck paths $P_1$ and $P_2$ satisfy that $P_2$ is above $P_1$ and the difference of the 
numbers of unit boxes below $P_1$ and $P_2$ is one. Then, we define $P_1\lessdot P_2$.
The cover relation of the poset $\mathcal{P}$ is simply given by this cover relation $\lessdot$. 
The poset on vertical Hermite histories can be viewed as a natural generalization of the poset $\mathcal{P}$.
\end{enumerate}
\end{remark}

Let $V_i$, $1\le i\le a$, be the vertical Hermite history of a trivial Dyck tiling $D_{i}$.
We say that $\mathbf{V'}:=(V'_1,\ldots,V'_{a})$ covers 
$\mathbf{V}:=(V_1,\ldots, V_{a})$ if and only if 
there exits a unique $j$ such that $V'_{j}$ covers $V_{j}$ as $(1,b)$-Dyck tilings 
and $\mathbf{V}$ and $\mathbf{V'}$ satisfy the condition (\ref{eq:condvij}).

Suppose that an $(a,b)$-Dyck tiling contains a non-trivial Dyck tile. In this case, we first replace 
non-trivial Dyck tiles with trivial Dyck tiles. Then, we decompose the new trivial Dyck tiling 
into $a$ $(1,b)$-Dyck tilings. To recover the non-trivial Dyck tiling, we replace trivial Dyck 
tiles in the corresponding $(1,b)$-Dyck tiling with a non-trivial Dyck tile such that it is 
compatible with the non-trivial Dyck tile in the $(a,b)$-Dyck tiling. 
For example, we have the following correspondence between a $(2,3)$-Dyck tiling and two $(1,3)$-Dyck tilings:
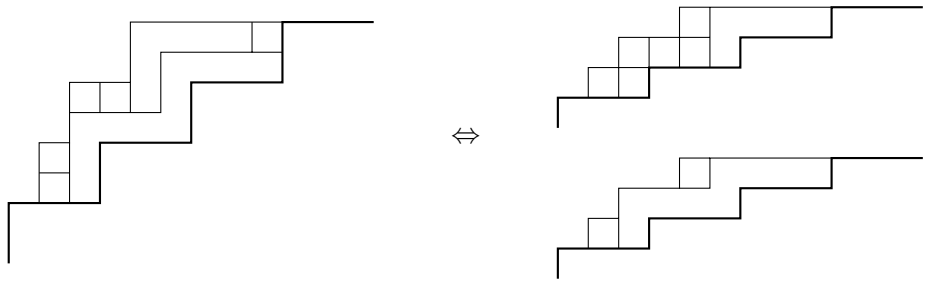
\begin{align*}
\tikzpic{-0.5}{[scale=0.4]
\draw[thick](0,0)--(0,2)--(3,2)--(3,4)--(6,4)--(6,6)--(9,6)--(9,8)--(12,8);
\draw(1,2)--(1,4)--(2,4)--(2,6)--(4,6)--(4,8)--(9,8);
\draw(2,2)--(2,4)(1,3)--(2,3)(2,5)--(5,5)--(5,7)--(9,7)(3,5)--(3,6)(4,5)--(4,6)(8,7)--(8,8);
}\qquad\Leftrightarrow\qquad
\tikzpic{-0.5}{[scale=0.4]
\draw[thick](0,0)--(0,1)--(3,1)--(3,2)--(6,2)--(6,3)--(9,3)--(9,4)--(12,4);
\draw(1,1)--(1,2)--(2,2)(2,1)--(2,3)--(5,3)--(5,4)--(9,4)(4,3)--(4,4)--(5,4);
\draw[thick](0,5)--(0,6)--(3,6)--(3,7)--(6,7)--(6,8)--(9,8)--(9,9)--(12,9);
\draw(1,6)--(1,7)--(3,7)--(3,8)(2,6)--(2,8)--(5,8)(4,7)--(4,9)--(9,9)(5,7)--(5,9);
}
\end{align*}
Note that the weight of the $(2,3)$-Dyck tiling is the same as the sum of the weights 
of the two $(1,3)$-Dyck tilings.

To replace trivial Dyck tiles with a non-trivial Dyck tile reduces the weight of the tiling.
We define the cover relation $D_1\lessdot D_2$ in the case where $D_2$ contains a non-trivial Dyck tile as follows.
Let $D'_i$, $i=1,2$, be the Dyck tilings consisting of only trivial Dyck tiles with the same top path as $D_{i}$.
Then, $D_1\lessdot D_2$ if and only if $D'_1\lessdot D'_2$, $\mathrm{wt}_{(a,b)}(D_2)=\mathrm{wt}_{(a,b)}(D_1)-1$ and 
$D_2$ can be obtained from $D_1$ by replacing Dyck tiles with a larger non-trivial tile.
If a Dyck tiling $D$ consists of non-trivial Dyck tiles of maximal size, there is no element which 
covers $D$.
We have a poset of Dyck tilings by the cover relation on vertical Hermite histories.
This poset has a natural grading by the sum of the weights of Dyck tiles.

\begin{example}
We consider the trivial $(2,3)$-Dyck tiling $D$ of size $3$ above the path $(U^2D^3)^3$,
whose vertical Hermite history is $(0,0,3,1,1,1,0,0,0)$.
Inside of this Dyck tiling, we have thirteen trivial Dyck tiling and 
one non-trivial Dyck tiling. Therefore, we have fourteen Dyck tilings 
which are lower than the tiling $D$.
Since $(a,b)=(2,3)$, we have two $(1,3)$-Dyck tilings whose vertical 
Hermite histories are $(002111000)$ and $(001000000)$. 

\begin{figure}[ht]
\begin{align*}
&\genfrac{(}{)}{0pt}{}{2111}{1000} && \genfrac{(}{)}{0pt}{}{1111}{1000} 
&&\genfrac{(}{)}{0pt}{}{1111}{0000} && \genfrac{(}{)}{0pt}{}{0111}{0000} 
&&\genfrac{(}{)}{0pt}{}{0011}{0000} && \genfrac{(}{)}{0pt}{}{0001}{0000}
&&\genfrac{(}{)}{0pt}{}{0000}{0000} \\[12pt]
& && \genfrac{(}{)}{0pt}{}{2110}{1000} && \genfrac{(}{)}{0pt}{}{1011}{1000} 
&& \genfrac{(}{)}{0pt}{}{1001}{1000} && \genfrac{(}{)}{0pt}{}{1000}{1000}
&& \genfrac{(}{)}{0pt}{}{1000}{0000} \\[12pt]
& && && && \genfrac{(}{)}{0pt}{}{1011}{0000} && \genfrac{(}{)}{0pt}{}{1001}{0000}
\end{align*}
\caption{Fourteen Dyck tilings for $(n,a,b)=(3,2,3)$.}
\label{fig:Dyck2314}
\end{figure}
In Figure \ref{fig:Dyck2314}, we list up all fourteen Dyck tilings.
We delete two $0$'s from left and three $0$'s from right in vertical Hermite histories.
The rank of the left-most element is zero, and that of the right-most element
is six.
The element corresponding to $\genfrac{(}{)}{0pt}{}{2110}{1000}$ is the unique non-trivial 
Dyck tiling.
Note that there is no element which covers this non-trivial Dyck tiling.
\end{example}

\bibliographystyle{amsplainhyper} 
\bibliography{biblio}

\end{document}